\documentclass[11pt]{article} 

\usepackage{changes}
\usepackage{enumerate}
\usepackage{authblk}
\usepackage{ulem,ifthen,xkeyval,pdfcolmk}

\definechangesauthor[color=red]{zl}

\usepackage{mathtools}
\usepackage{mathrsfs}
\usepackage{amssymb}

\usepackage{graphics}
\usepackage{graphicx}

\usepackage{amsmath}
\usepackage{amsthm}
\usepackage{amsfonts}
\usepackage[english]{babel}
\usepackage[margin=0.8in]{geometry}
\usepackage[colorlinks=true, linkcolor=blue]{hyperref}
\usepackage{bbm}
\usepackage{verbatim}
\usepackage{mathtools}
\usepackage{apptools}
\usepackage{layout}
\usepackage[makeroom]{cancel}
\usepackage{bbm}
\usepackage{enumerate,enumitem}

\newcommand{\mE}{\mathcal{E}}
\newcommand{\F}{\mathcal{F}}
\newcommand{\mC}{\mathcal{C}}
\newcommand{\mF}{\mathcal{F}}
\newcommand{\mD}{\mathcal{D}}

\newcommand{\R}{\mathbb{R}}

\newcommand{\LX}{{L^2(X,\,\mu)}}

\newcommand{\abs}[1]{\left|#1\right|}       
\newcommand{\norm}[1]{||#1||}                  
\newcommand{\ang}[1]{\left<#1\right>}       
\newcommand{\brak}[1]{\left(#1\right)}      
\newcommand{\crl}[1]{\left\{#1\right\}}     

\DeclareMathOperator*{\esssup}{ess\,sup}
\DeclareMathOperator*{\osc}{\text{osc}}
\DeclareMathOperator*{\essinf}{ess\,inf}

\newtheorem{theorem}{Theorem}[section]

\newtheorem{definition}[theorem]{Definition}
\newtheorem{proposition}[theorem]{Proposition}
\newtheorem{corollary}[theorem]{Corollary}
\newtheorem{lemma}[theorem]{Lemma}
\newtheorem{remark}[theorem]{Remark}

\numberwithin{equation}{section}

\title{Extension method in Dirichlet spaces with sub-Gaussian estimates and applications to regularity of jump processes on fractals} 

\author{Fabrice Baudoin,Quanjun Lang,Yannick Sire}

\date{}

\AtEndDocument{\bigskip{\footnotesize%
  (F.~Baudoin) \textsc{Department of Mathematics, Aarhus University, Aarhus, DK, 8000} \par  
  \textit{E-mail address}:  \texttt{fbaudoin@math.au.dk} \par
  \addvspace{\medskipamount}
   (Q. ~Lang) \textsc{Department of Mathematics, Duke University, Durham, NC, 27710}  \par
  \textit{E-mail address}:  \texttt{quanjun.lang@duke.edu} \par
  \addvspace{\medskipamount}
  (Y.~Sire) \textsc{Department of Mathematics, Johns Hopkins University, Baltimore, MD, 21218} \par
  \textit{E-mail address}: \texttt{ysire1@jhu.edu}
}}

\begin{document}
\maketitle

\begin{abstract}
We investigate regularity properties of some non-local equations defined on Dirichlet spaces equipped with sub-Gaussian estimates for the heat kernel associated to the generator. We prove that weak solutions for homogeneous equations involving pure powers of the generator are actually H\"older continuous and satisfy an Harnack inequality. Our methods are based on a version of the Caffarelli-Silvestre extension method which is valid in any Dirichlet space and our results complement the existing literature on  solutions of PDEs on classes of Dirichlet spaces such as fractals.  
\end{abstract}

\tableofcontents

\section{Introduction}

The last two decades have seen an important amount of work whose aim is to realize the powers of some operators $L$ in terms of a suitable extension, in order to deduce properties of solutions as a by-product of solutions in the extension. This method is very powerful to deal with equations involving powers of $L$ but also, in the other direction, to understand properties of equations in a "lifted" space of the type $X \times \R^+$ using tools of purely non-local nature on the space $X$. This approach has been undertaken in a PDE framework in the original seminal paper of Caffarelli and Silvestre \cite{Caffarelli_2007}. Since then, a very substantial amount of literature has been devoted to this type of "dictionary" between equations satisfied in $X\times \R^+$ and the related equation on $X$.

When the operator $L$ is second-order in divergence-form, for instance, the extension appears to be a differential operator, and classical tools from PDEs allow to get (or recover) several results on the operator $(-L)^s$ (for $0<s<1$) such as regularity estimates and fine properties of solutions of an associated PDE. Functions of $L$ are, of course, multipliers in the sense of harmonic analysis, and one can then get new proofs of some classical results (see, e.g.  \cite{LenzSchi,brazke,BBCS} and references therein). 

Powers of some suitable $L$ are a subclass of generators of L\'evy processes, and some results which can be proved via probabilistic techniques can be recovered through PDE ones. We refer the reader to the book \cite{bertoin} for an extensive study of L\'evy processes. The idea of using results in the extension to obtain some on the boundary is already presented in \cite{Caffarelli_2007}. Since the extension by Caffarelli and Silvestre involves $A_2$ weights and a rather satisfactory basic theory for those is available in the literature (see e.g., \cite{FKS,FJK,FJK2} ), one deduces a wealth of results such as Harnack principles and regularity of solutions for equations of the type $(-\Delta)^s u=0$ with $s \in (0,1)$. Notice that similar results have been obtained for powers of the heat operator in \cite{2019arXiv191105619B,BG,STparab,AudritoTerracini}.  In \cite{ChenKumagaiWang, ChenKumagaiWang2,ChenKumagaiWang3}, Chen, Kumagai and Wang developed a far reaching theory of parabolic Harnack inequalities for {\sl non-local Dirichlet forms}. Such Dirichlet forms are typically associated to generators like $(-\Delta)^s$ for $s \in (0,1)$.  Building on many important works over many years, the authors prove some precise equivalences between some versions of the parabolic Harnack inequality for the generator and various sharp heat kernel bounds (among other results). In the current paper, our approach is to use the interplay between the extended problem (and properties of weak solutions on it) and the trace of its solutions which are by construction weak solutions of $L^s$. This allows to prove many results for $L^s$ out of results on the corresponding extended problem. Of course, as it is now well understood, such technique works only for pure powers of suitable generators and one cannot deal as in \cite{ChenKumagaiWang, ChenKumagaiWang2,ChenKumagaiWang3} with general non-local Dirichlet forms.

In \cite{stingatorrea}, Stinga and  Torrea develop a framework for powers of a general non-negative self-adjoint operator $L$ on $L^2(\Omega, d\eta)$, with $\Omega$ being an open set in $\R^n$ and $d\eta$ a positive measure. Their approach, based on transference principles, only requires the existence of a heat kernel and semi-group theory. It covers a much wider class of setups than the approach developed by Caffarelli and Silvestre in their original paper. Similarly, Kwasnicki and Mucha \cite{Kwasnicki2018} discussed the extension problem for complete Bernstein functions of the Laplacian, other than just fractional powers. 

One of the goals of the paper is to provide a unified framework of the extension theorem in the context of Dirichlet spaces. Formally, the approach of Stinga and Torrea is extremely flexible and provides some classical notions on Dirichlet spaces are known,  their proof adapts {\sl verbatim}. It is also clear that the construction of  Kwasnicki and Mucha in our context (see \cite{Kwasnicki2018} for the Euclidean case) allows functions of the generator can also be performed. 

Another contribution of this paper is to investigate H\"older regularity and Harnack inequalities for weak solutions of $(-L)^s u=0$, $0<s<1$, on $X$ where $X$ is a Dirichlet space, with a generator $L$ having sub-Gaussian estimates. Examples of such a framework include a large class of fractals. {Such a result could be deduced  from \cite{CKTW}. However, to broaden the scope of the result and appeal to PDE arguments, we provide a proof using an extension method, in the spirit of Caffarelli and Silvestre \cite{Caffarelli_2007} or Stinga and Torrea \cite{stingatorrea}. }

We now describe our setup and refer the reader to next section for precise definitions. Consider a Dirichlet space $(X,d,\mu,\mE, \mF)$. Assume that it is strongly local and regular. Let $-L$ be the associated generator to the form $\mE$ and assume that the associated  heat kernel $p_t$ satisfies sub-Gaussian estimates, namely  
	\begin{align}\label{eq:subGaussian-intro}
		\frac{c_1}{t^{d_H/d_W}}\text{exp}\brak{-c_2\brak{\frac{d(x,y)^{d_W}}{t}}^{\frac{1}{d_W-1}}} \leq p_t(x,y) \leq \frac{c_3}{t^{d_H/d_W}}\text{exp}\brak{-c_4\brak{\frac{d(x,y)^{d_W}}{t}}^{\frac{1}{d_W-1}}}
	\end{align}
	holds for $\mu\times \mu$-a.e. $(x,y)\in X \times X$ and each $t\in(0, +\infty)$, where $c_1, c_2, c_3, c_4 > 0$ and $d_H >0$, $d_W \in [2, + \infty)$ are constants independent of $x, \ y$ and $t$. The constant $d_H$ is the so-called Hausdorff dimension of $X$ and $d_W$ is the walk dimension. 
	
%

Our main result is the following theorem on regularity properties of solutions of $(-L)^s u =0$. 
\begin{theorem}\label{thm:EHI_weak_solution_Ls}
	Suppose the sub-Gaussian heat kernel estimates above \eqref{eq:subGaussian-intro} hold on the Dirichlet space $(X,d,\mu,\mE, \mF)$. Consider a non-negative weak solution $f$ of $(-L)^sf = 0$ where $-L$ is the generator of $\mE$ and the parameter range is $s \in (0,1)$. 
	
	Then the following holds: the function $f$ satisfies a {\sl global} Harnack inequality 
	\begin{equation}\label{eq:X_elliptic_HI}
		\sup_{B(x_0, R)} f \leq C \inf_{B(x_0, R)}f
	\end{equation}
	for any $x_0\in X$ and $R > 0$. Moreover, the following H\"{o}lder estimate holds, 
	\begin{equation}
		\abs{f(x) - f(y)} \leq C \left(\frac{d(x, y)^{2/d_W}}{R}\right)^\alpha \osc_{B(x_0, R)}f
	\end{equation}
	for $\mu$-almost all $x, y\in B(x_0, R)$.
\end{theorem}

The previous theorem has been proved in \cite{gianmarco} in the case of Cheeger spaces. Our aim here is to provide a proof of H\"older regularity and Harnack inequality in a Dirichlet setting encompassing fractals.

We combine several techniques, but the main result, the Harnack inequality, is not a direct consequence of the extension, as in the more regular space version. This fact is due to that in our setting the lifted space after the extension does not support sub-Gaussian heat kernel estimates any more. This means that the Harnack inequality cannot be obtained via standard methods in the extended space since, combined to the volume doubling property, it would imply Gaussian estimates on $X$, hence a contradiction with our standing assumptions. Instead, in the extended space, we derive a {\sl partial or anisotropic} Harnack inequality on products of balls with anisotropic scalings for the variables. Passing to the trace it leads the desired Harnack inequality for the {\sl elliptic} original PDE.  {We would like to emphasize that this partial Harnack inequality reflects the geometry of $X$ {\sl vs } the one of its lift $X \times (0,\infty)$. }

As already alluded, results like Theorem \ref{thm:EHI_weak_solution_Ls} pertain to a theory which has been developed in greater generality in \cite{ChenKumagaiWang, ChenKumagaiWang2,ChenKumagaiWang3}. In this paper, the authors work directly at the level of the non-local Dirichlet form and make full use of the probabilistic properties of the jump process associated with the Dirichlet form. This creates very serious difficulties and the results they prove build on many previous contribution, for instance \cite{Barlow2006Stability,BBK}. On the other hand, our result can be obtained through simpler and perhaps arguments, at the price of restricting the type of jump process under consideration, and sheds some light on the dichotomy between the {\sl local} Dirichlet form on the extended space and the {\sl non-local} on $X$. In view of this, the inhomogeneous partial Harnack inequality happens to be the correct object in this setting.

Finally, we would like to mention that,  there are other strategies to prove H\"older regularity and/or Harnack inequalities in the case of our class of jump processes. Indeed,  one can consider directly the non-local equation expressing the operator in integral form so that in this case we fall into the framework of \cite{ChenKumagaiWang}, but  using PDE techniques recently developed to handle directly non-local equations at the global level, i.e. without using the extension and incorporating the contributions of long range and jump effects of the process. In our general framework, this approach is somehow more involved to implement. Indeed a straightforward computation based on the spectral theorem shows that under natural assumptions, one can show that an integral expression of $(-L)^s$ can be obtained as 
\[
(-L)^s f (x) =P.V. \int_X K(x,y) (f(y)-f(x)) d\mu(y)
\]
with
\[
K(x,y)=\int_0^{+\infty} p_t(x,y)  \frac{dt}{t^{1+s}}
\]
and where $p_t$ is the heat kernel associated to $-L$. Under additional assumptions on the heat kernel $p_t$ and using the previous formula for the Kernel $K$, one  would obtain some {\sl a priori} bounds (and regularity) on $K$, which are reminiscent of some ellipticity conditions.  This opens the way to use arguments as in e.g. \cite{SilvestreIUMJ,dCKP1,dCKP2}. A more precise version of the previous computation can be found in \cite{ACM} in the context of some Riemannian manifolds . It is however not completely clear how to derive the formula in \cite{ACM} in our general set up since it relies on the Hadamard parametrix, which is a very Riemannian construction. We refrain from using a purely nonlocal approach since our method via extension offers perspectives on the different behaviours in the lifted space.

\section{Preliminaries on Dirichlet Spaces}\label{sec:Dirichlet_spaces}


\subsection{Basic definitions on Dirichlet spaces}
Here we provide an introduction to Dirichlet spaces. We refer the reader to \cite{DirichletFormsandSymmetricMarkovProcesses} for more details.
Let $(X,d)$ be a locally compact, complete, separable, metric space equipped with a Radon measure $\mu$ with full support. Let $(\mE,\mF = \mD(\mE))$ be a densely defined, non-negative, symmetric bilinear form on $\LX$. Note that 
$$(u,v)_\mF = (u,v)_{\LX} + \mE(u,v)$$
is an inner product on $\mF$. 
Then we can define the norm on $\mF$ by Cauchy-Schwarz inequality, 
$$\norm{u}_{\mF} = \brak{\mE(u,u) + \norm{u}_{\LX}^2}^{1/2}.$$
We say $\mE$ is $closed$ if $\mF$ is complete with respect to the norm $\norm{\cdot}_\mF$.
Given $\mE$ is closed, we say it is $Markovian$ if 
$$
u \in \mF, \,v \text{ is a normal contraction of  } u \Rightarrow v \in \mF, \,\mE(v,v) \leq \mE(u,u).$$
Here a function $v$ is called  a  $normal\  contraction$ of a function $u$, if 
$$
\abs{v(x)-v(y)} \leq \abs{u(x)-u(y)},\ \forall x,y\in X, \ \abs{v(x)}\leq \abs{u(x)},\ \forall x \in X.
$$
\begin{definition}
We say $(\mE,\mathcal{F} = \mathcal{D}(\mE))$ is a Dirichlet form on $L^2(X,\mu)$, if $\mE$ is a densely defined, closed, non-negative,  symmetric and Markovian bilinear form on $L^2(X,\mu)$. 
\end{definition}

By \cite[Theorem 1.3.1]{DirichletFormsandSymmetricMarkovProcesses}, we can define the generator for a Dirichlet form.
\begin{definition}
There is a one-to-one correspondence between the family of closed symmetric forms $\mE$ on $L^2(X, \mu)$ and the family of non-positive definite self-adjoint operators $L$ on $L^2(X, \mu)$. The correspondence is determined by
\begin{align*}
    \begin{cases}
    \mathcal D\brak{\mE} = \mathcal{D}\brak{\sqrt{-L}}\\
    \mE(u,v) = (\sqrt{-L}u,\sqrt{-L}v)
    \end{cases}
\end{align*}
$L$ is called the generator of the Dirichlet form $\mE$.
\end{definition}
Since $L$ is non-positive self-adjoint, by the spectral theorem, there exists a unique spectral measure $dE(\lambda)$, such that 
$$ -L = \int_0^\infty \lambda dE(\lambda).$$
This formula is understood in the following sense: for any functions $f,g\in \mathcal{D}\brak{L}$, we have 
$$\ang{-Lf, g} = \int_0^\infty \lambda dE_{f,g}(\lambda),$$
where $dE_{f,g}$ is a well-defined Radon measure on $X$.
In particular, for any non-negative continuous function $\phi$ on $[0,\infty)$, we can define
\begin{align}
    \begin{cases}
    \phi(-L) = \int_0^\infty \phi(\lambda)dE(\lambda),\\
    \mathcal{D}\brak{\phi(-L)} = \crl{u\in \mF : \int _0 ^\infty \phi(\lambda)^2 dE_{u,u}(\lambda)<\infty}. 
    \end{cases}
\end{align}
We will in particular, consider fractional powers of the generator $L$, which are then defined by
\begin{align}
	\ang{(-L)^sf, g } =   \int_0^\infty \lambda^s dE_{f,g}(\lambda),
\end{align}
and the domain of $(-L)^s$ is given by 
\begin{align}
	\mathcal{D} \brak{(-L)^s} = \crl{u\in \mF : \int _0 ^\infty \lambda^{2s} dE_{u,u}(\lambda)<\infty}. 	
\end{align}
Similarly, we can define the semigroup associated with $L$ as $P_t = e^{tL}$. This defines a strongly continuous semigroup on $L^2(X,\mu)$ which is called the heat semigroup associated with the Dirichlet form $\mathcal{E}$. Though originally defined on $L^2(X,\mu)$, it turns out that for every $1 \le p \le +\infty$ the semigroup can be extended from $L^2(X,\mu) \cap L^p(X,\mu)$ to $L^p(X,\mu)$, defining then for all $p$ a contraction semigroup $L^p(X,\mu) \to L^p(X,\mu)$. This semigroup is Markovian, i.e. if $0 \le f \le 1$ $\mu$-a.e., then for every $t \ge 0$,  we have $\mu$-a.e. $0 \le P_t f \le 1$.  Associated to $P_t$, there is always a heat kernel measure, which is a family of measures $p_t(x,dy)$ defined on $X$ for every $x \in X$ and that satisfies
\[
P_tf(x)=\int_X f(y) p_t(x,dy), \quad f \in L^\infty (X,\mu), t >0.
\]
If $p_t(x,dy)=p_t(x,y) d\mu(y)$ for some measurable function $p: (0+\infty) \times X \times X \to \mathbb{R}$, then this function $p_t(x,y)$ is called the heat kernel of the Dirichlet form.

We denote by $C_c(X)$ the space of all continuous functions with compact support in $X$ and $C_0(X)$ its closure with respect to the supremum norm. The following definitions hold:

\begin{definition}
A $core$ of $\mE$ is a subset $\mathcal{C}$ of $\F \cap C_0(X)$ such that $\mathcal{C}$ is dense in $\F$ with the norm $\norm{\cdot}_\F$ and dense in $C_0(X)$ with the supremum norm.

\begin{enumerate}
\item A Dirichlet form $\mE$ is called $regular$ if it admits a core. 
\item A Dirichlet form $\mE$ is called $local$ if for any $u, v \in \F$  with disjoint compact support, then $\mE(u,v) = 0$. $\mE$ is called non-local if this property does not hold.
\item A Dirichlet form $\mE$ is called strongly  local if for any $u,v\in\F$ with compact support, $v$ is constant on a neighborhood of the support of $u$, then $\mE(u,v) = 0$.
\end{enumerate}

\end{definition}

We now assume throughout the paper that $(\mE,\mathcal{F})$ is a strongly local regular Dirichlet form on $\LX$.   Since  $\mE$ is regular, the following holds

\begin{definition}
Suppose $\mE$ is a regular Dirichlet form, for every $u,v\in \mathcal F\cap L^{\infty}(X)$, the energy measure $\Gamma (u,v)$ is defined  through the formula
 \[
\int_X\phi\, d\Gamma(u,v)=\frac{1}{2}[\mE(\phi u,v)+\mE(\phi v,u)-\mE(\phi, uv)], \quad \phi\in \mathcal F \cap C_c(X).
\]
\end{definition}
Note that $\Gamma(u,v)$ can be extended to all $u,v\in \mathcal F$ by truncation (see \cite[Theorem 4.3.11]{SymmetricMarkovProcessesTimeChangeandBoundaryTheory}). According to Beurling and Deny~\cite{beurling1958}, one has then   for $u,v\in \mathcal{F}$
\[
\mathcal E(u,v)=\int_X d\Gamma(u,v)
\]
and $\Gamma(u,v)$ is a signed Radon measure.
%
%
%
%
%

\subsection{Tensorization of Dirichlet spaces}\label{tensorization Dirichlet space}

Let  $(\tilde{\mE},\tilde{\mathcal{F}})$ be  a regular Dirichlet form on $L^2(Y,m)$ where $(Y,m)$ is a metric measure space that satisfies the same conditions as $(X,\mu)$. One can  define a Dirichlet form on the product space $(X \times Y, \mu \otimes m)$ as follows. 

Denote by $\tilde{L}$ the generator of the Dirichlet form $(\tilde{\mE},\tilde{\mathcal{F}})$. One defines an operator $\mathcal{L}$ on $\mathcal{D}(L) \otimes \mathcal{D}(\tilde{L})$ by
\[
\mathcal{L} ( f\otimes g)= (Lf ) \otimes g + f \otimes \tilde{L}g
\]
By linearity, this operator is extended to the pre-domain
\[
\left\{ \sum_{i=1}^n f_i \otimes g_i : n \in \mathbb{N}, f_i \in  \mathcal{D}(L), g_i \in  \mathcal{D}(\tilde L)  \right\}
\]
The operator $\mathcal{L}$ is then non-positive and symmetric. It admits therefore a minimal self-adjoint extension $(\mathcal{L}, \mathcal{D}(\mathcal{L}))$. Using spectral theorem as before, $(\mathcal{L}, \mathcal{D}(\mathcal{L}))$ is the generator of a strongly continuous semigroup $P^{X,Y}_t$ on $L^2(X \times Y, \mu \otimes m)$. This semigroup is Markovian and can be represented as
\[
P^{X,Y}_t f (x,y) =\int_X \int_Y f(z_1,z_2) \tilde{p}_t(y, dz_2) p_t(x, dz_1), \quad t >0, f \in L^\infty(X\times Y, \mu \otimes m)
\] 
where $\tilde{p}_t$ is the heat kernel measure of $(\tilde{\mE},\tilde{\mathcal{F}})$. The Dirichlet form $(\mathcal{E}^{X,Y}, \mathcal{F}^{X,Y}:=\mathcal{D}(\sqrt{-\mathcal{L}}))$ on the product space is  the Dirichlet form associated with this semigroup:
\[
\mathcal{E}^{X,Y}(f,f)=\lim_{t \to 0} \frac{1}{t} \langle f -P^{X,Y}_t f, f\rangle_{L^2(X \times Y, \mu \otimes m)}=\| \sqrt{-\mathcal{L}} f \|^2_{L^2(X \times Y, \mu \otimes m)}.
\]
Clearly, for $f \in \mathcal{F}$, $g \in \tilde{\mathcal{F}}$, $f\otimes g \in \mathcal{F}^{X,Y}$ and
\[
\mathcal{E}^{X,Y}(f\otimes g,f\otimes g)=\mathcal{E}( f) \| g \|_{L^2(Y,m)}^2 +\| f \|_{L^2(X,\mu)}^2\tilde{\mathcal{E}}(g).
\]
Finally, the Dirichlet form  $(\mathcal{E}^{X,Y}, \mathcal{F}^{X,Y})$ is seen to be regular, because if $\tilde{\mathcal{C}}$ is a core for $(\tilde{\mE},\tilde{\mathcal{F}})$ then
\[
\left\{ \sum_{i=1}^n f_i \otimes g_i : n \in \mathbb{N}, f_i \in \mathcal{C}, g_i \in \tilde{\mathcal{C}} \right\}
\]
is a core for $(\mathcal{E}^{X,Y}, \mathcal{F}^{X,Y})$.

\subsection{Weak solutions associated to generators}

We introduce the following spaces. For a domain $\Omega \subset X$, define
\begin{align}
&\mathcal{F}_{loc}(\Omega)=\left\{ f \in L^2_{loc}(\Omega), \text{ for every relatively compact } V \subset \Omega , \, \exists f^* \in \mathcal{F}, \, f^*_{\mid V}=f_{\mid V}, \, \mu \, a.e. \right\}, \label{def:F_loc}\\
&\mathcal{F}_{c}(\Omega)=\left\{ f \in \F: \text{The essential support of $f$ is compact in $\Omega$}\right\},\label{def:F_c}\\
&\mathcal{F}^0(\Omega)= \text{The closure of $\F_c(\Omega)$ with respect to the norm } \|f\|_\F.\label{def:F^0}
\end{align}

\begin{remark}
For $f,g \in \mathcal{F}_{loc}(\Omega)$, on can define $\Gamma (f,g)$ locally by $\Gamma (f,g)_{\mid V}=\Gamma( f^*_{\mid V},g^*_{\mid V}) $.

\end{remark}

\begin{definition}\label{def:harmonic_functions}
For a domain $\Omega \subset X$, a  function $f \in \mathcal{F}_{loc}(\Omega)$ is called harmonic in $\Omega$ if for every function $h \in \mathcal{F}$ whose essential support is included in $\Omega$, one has
\[
\mE (f, h)=0.
\]
In particular, $f$ is also called a $weak \ solution$ for $Lf = 0$ in $\Omega$, since $\mE(f, h) = \ang{-Lf, h}$.
\end{definition}

Similarly, we have
\begin{definition}\label{def:weak_solution_Ls}
For a domain $\Omega \subset X$, a function $f \in \mathcal{D}\brak{(-L)^s}$ is called a weak solution for $(-L)^s f = 0$, if for every function $h \in \mathcal{F}$ whose essential support is included in $\Omega$, one has
\[
 \ang{(-L)^sf, h}=0.
\]
\end{definition}

\subsection{Sub-Gaussian heat kernel estimate}

%
%
%

{The following definition is our main assumption on the Dirichlet space $X$. We refer to  \cite{barlow2012equivalence} and  \cite{grigor2003heat,MR2962091,hkmma} for more details.}
\begin{definition}\label{def:heat_kernel_estimate}
	Let $(X, d, \mu, \mE, \F)$ be a Dirichlet space with generator $L$. Let $P_t$ be the heat semi-group associated with $L$. We say the kernel $p_t(x,y)$ of $P_t$ satisfies sub-Gaussian estimates if
	\begin{align}\label{eq:subGaussian}
		\frac{c_1}{t^{d_H/d_W}}\text{exp}\brak{-c_2\brak{\frac{d(x,y)^{d_W}}{t}}^{\frac{1}{d_W-1}}} \leq p_t(x,y) \leq \frac{c_3}{t^{d_H/d_W}}\text{exp}\brak{-c_4\brak{\frac{d(x,y)^{d_W}}{t}}^{\frac{1}{d_W-1}}}
	\end{align}
	holds for $\mu\times \mu$-a.e. $(x,y)\in X \times X$ and each $t\in(0, +\infty)$, where $c_1, c_2, c_3, c_4 > 0$, $d_H >0$ and $d_W \in [2, + \infty)$ are constants independent of $x, \ y$ and $t$.
\end{definition}

{We point out that most of our results will be subject to assuming Estimates \eqref{eq:subGaussian}. By Theorem $3.1$ in \cite{hkmma}, the estimates \eqref{eq:subGaussian} imply that the measure $\mu$ satisfies the following David-Ahlfors bounds  } 
\begin{equation}\label{eq:ball_size}
	\mu(B(x, r)) \asymp  r^{d_H}
\end{equation}
where $B(x, r)$ is the open metric ball with center $x$ and radius $r$ in $(X, d)$. A notable consequence of the David-Ahlfors regularity of the measure is the volume doubling property:
\[
\mu(B(x,2 r)) \le C \mu(B(x, r)),
\]
and a classical argument gives then that there exists $C_{VD}, \gamma >0$ such that 
\begin{equation}\label{eq:VD_gamma}
    V(x, R) \leq C_{VD} V(y, r) \brak{\frac{d(x, y) + R}{r}}^\gamma, \text{ for all } x, y\in X, 0<r\leq R.
\end{equation}
where $V(x,R)$ is the volume of the ball $B(x,R)$.

There are many examples of Dirichlet spaces where the heat kernel satisfies the estimates \eqref{def:heat_kernel_estimate}. In particular a large family of examples of Dirichlet spaces for which sub-Gaussian estimates hold are the p.c.f. fractals like the unbounded Vicsek set (see \ref{fig-V}), the unbounded Sierpi\'nski gasket (see \ref{fig-SG}) or the infinite Sierpi\'nski carpet (see \ref{fig-SC}). We refer to the lecture notes  \cite{MR1668115} by M. Barlow where those examples are discussed in details. 

\begin{figure}[htb]\centering
\includegraphics[trim={63 83 180 201},clip,height=0.20\textwidth]{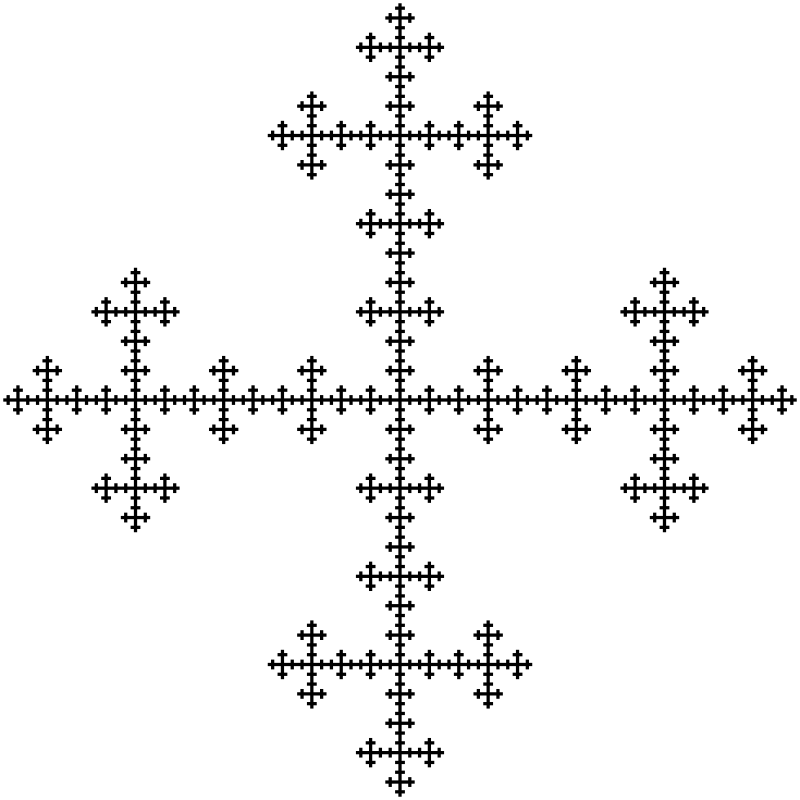}
\caption{A part of an infinite, or unbounded, Vicsek set.}\label{fig-V} 
\label{fig-Vicsek}
\end{figure}

 \begin{figure}[htb]\centering
 	\includegraphics[trim={83 0 169 63},clip,height=0.20\textwidth]{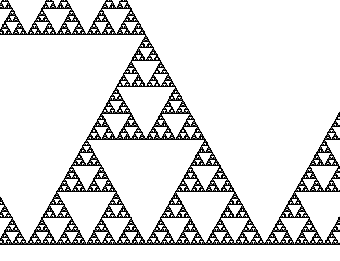}
 	\caption{A part of an infinite, or unbounded, Sierpinski gasket.} \label{fig-SG}
 \end{figure}

  \begin{figure}[htb]
  	\centering
  	\includegraphics[trim={8 23 38 70},clip,height=.20\textwidth] {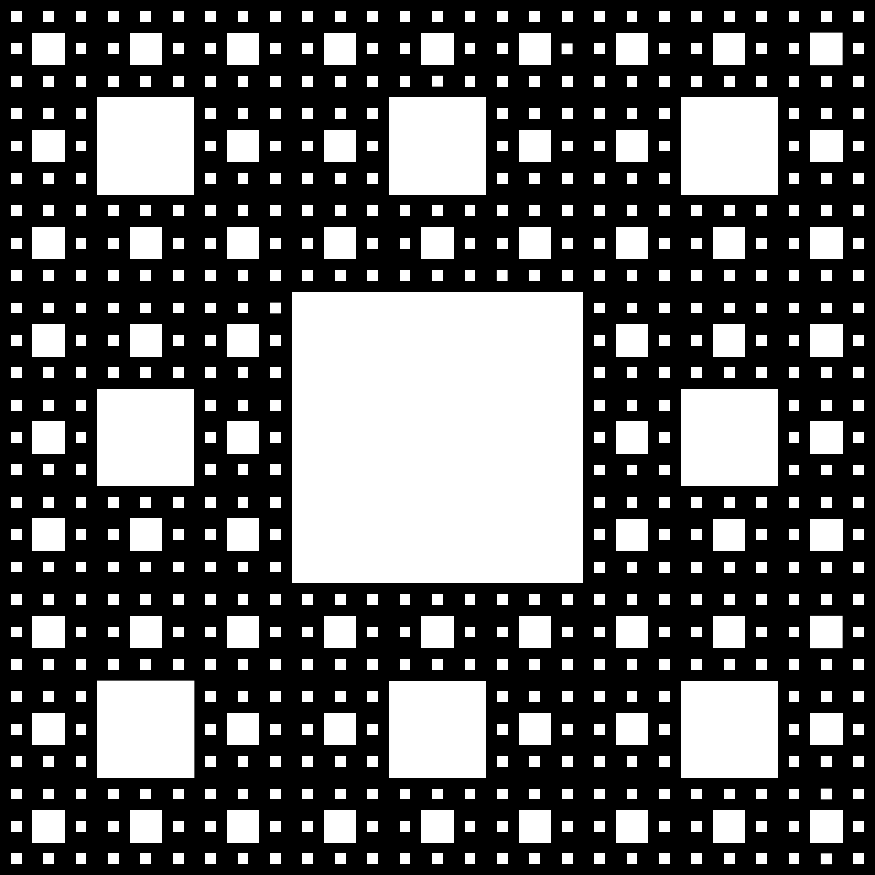}
  	\caption{A part of an infinite, or unbounded, Sierpinski carpet.}\label{fig-SC}
  \end{figure}

\subsection{Caloric functions}
We need to define what it means that a function $u(t, x)$ is a caloric function in a cylinder $I\times \Omega$, where $I$ is an interval in $\R$ and $\Omega$ is an open subset of $X$. We start by the weak differentiability.

\begin{definition}
	Let $I$ be an interval in $\R$. We say that a function $u:I\to L^2(\Omega)$ is weakly differentiable at $t_0 \in I$ if for any $f \in L^2(\Omega)$, the function $(u(t), f)$ is differentiable at $t_0$, (where the brackets stand for the inner product in $L^2(\Omega)$). That is, the following limit exists.
	\begin{equation}
		\lim_{t \to t_0} \brak{\frac{u(t) - u(t_0)}{t - t_0}, f}
	\end{equation}
\end{definition}

And then,
\begin{definition}
	Let $(X, d, \mu, \mE, \mF)$ be a Dirichlet space. Consider a function $u:I\to \mF$, and let $\Omega$ be an open subset of $X$. We say that $u$ is a $subcaloric$ function in $I\times \Omega$ if $u$ is weakly differentiable in the space $L^2(\Omega)$ at any $t \in I$ {(in the sense of  the previous definition)} and for any non-negative $f \in \mF_{loc}(\Omega)$ and for any $t \in I$, 
	\begin{equation}\label{eq:subcaloric}
		(u', f) + \mE(u, f) \leq 0.
	\end{equation}
	Similarly, one defines the notions of supercaloric functions and caloric functions. For the later, the inequality in \eqref{eq:subcaloric} becomes equality for all $f \in \mF_{loc}(\Omega)$.
\end{definition}

We have the super-mean value inequality for caloric functions, and the proof can be found in \cite[Corollary 2.3]{barlow2012equivalence}.
\begin{corollary}[Super-mean value inequality]\label{cor:super_mean_value}
	Let $f \in L^2_+(\Omega)$ and $u$ be a non-negative supercaloric function in $(0, T) \times \Omega$ such that $u(t, \cdot)\xrightarrow{L^2(\Omega)}f$ as $t \to 0$. Then for any $t \in (0, T)$, 
	\begin{equation}
		u(t, \cdot) \geq P_t^\Omega f \text{ in } \Omega.
	\end{equation}
	In particular, for all $0 < s < t < T$, 
	\begin{equation}\label{eq:SMV_2}
		u(t, \cdot) \geq P_{t-s}^\Omega u(s, \cdot) \text{ in } \Omega
	\end{equation}
\end{corollary}

\section{Extension Theorem on Dirichlet Spaces}\label{sec:DTN}
\subsection{Pure Fractional powers} 
In this section, we establish an extension theorem for the fractional power $(-L)^s$ of the generator of any Dirichlet form. This theorem is a direct generalization of the extension theorem in \cite{stingatorrea}. Here the fractional power $(-L)^s$ and its domain $\mathcal{D}((-L)^s)$ are defined by the spectral theorem for $0<s<1$ as introduced in Section \ref{sec:Dirichlet_spaces}.

We introduce some notations that will appear in the theorem. For any $-1<a<1$, we consider the space $\mathbb{R}$ endowed with the measure $d\nu_a = \abs{y}^a dy.$ The $Bessel \ operator$ is defined on $\brak{\mathbb{R}, d\nu_a}$ as 
\begin{align}\label{eq:Ba}
\mathcal B_a = \frac{\partial^2 }{\partial y^2} + \frac{a}{y} \frac{\partial }{\partial y}.
\end{align}
Notice that the space $\brak{\mathbb{R}, d\nu_a}$ is in an homogeneous space since the measure $d\nu_a$ is a doubling measure. The function $y \to \abs{y}^a$ (and its inverse), which is $L^1_{loc}$, is even an $A_2$ weight on $\mathbb R$,  i.e. $w$ is an $A_2-$weight if  there exists a constant $C>0$ such that any interval $I \subset \mathbb R$, the following holds 
$$
\int_I w \int_I w^{-1} \leq C |I|^2 . 
$$

Let $X_a$ denote the space $X \times \R$ endowed with the product measure $d\mu_a = d\mu d\nu_a$. We will also use $X^+_a$ to denote the space $X\times (0, +\infty)$ with the measure $d\mu_a$. The following theorem describes the extension  properties of $(-L)^s$. 

\begin{theorem}\label{thm:extension_theorem}
    Let  $f \in \mathcal{D}((-L)^s)$ and let $P_t = e^{tL}$ denote the semigroup generated by $L$. Consider the boundary value problem
    \begin{equation}\label{eq:entension_equation}
    \begin{cases}
    L_a U = (L + \mathcal B_a) U = 0\ \ \ \ \ \ \ \text{in}\ X_a^+,
    \\
    U(\cdot ,0) = f ,
    \end{cases}
    \end{equation}
    where $a = 1- 2s$, we have
    \begin{enumerate}[label=(\roman*)]
        \item The function 
\begin{align}
	U (\cdot ,y)=\frac{1}{\Gamma(s)} \int_0^{+\infty}   (P_t (-L)^s f)\cdot e^{-\frac{y^2}{4t}} \frac{dt}{t^{1-s}}.
\end{align}
is a weak solution of \eqref{eq:entension_equation}. 

        \item The following Poisson formula holds in the weak sense
\begin{align}\label{eq:Possion_formula}
U (\cdot ,y)= \frac{y^{2s}}{2^{2s}\Gamma(s) }\int_0^{+\infty} (P_t f)\cdot e^{-\frac{y^2}{4t}} \frac{dt}{t^{1+s}}.
\end{align}

        \item The Dirichlet-to-Neumann condition holds, i.e. the following equation holds weakly
\begin{align}\label{eq:D2N}
(-L)^s f = 
- \frac{2^{2s-1} \Gamma(s)}{\Gamma(1-s)} \underset{y\to 0^+}{\lim} y^a \frac{\partial U}{\partial y}(\cdot,y)
\end{align}
\end{enumerate}
The solution $U$ is called the $s$-harmonic extension of $f$.

\end{theorem}

The proof is almost the same as in \cite{stingatorrea} and we provide a sketch of it below. Our goal here is to provide the most general framework available to deal with equations seen in the extension. The previous result covers in particular the following examples, some of which were previously considered in the literature:
\begin{itemize}
\item Complete Riemannian manifolds with non-negative Ricci curvature or more generally RCD$(0,\infty)$ spaces in the sense of Ambrosio-Gigli-Savar\'e \cite{ambrosio2014}.
\item Carnot groups and other complete sub-Riemannian manifolds satisfying a generalized curvature dimension inequality (see \cite{BAUDOIN20122646,BK}). The case of Heisenberg groups was considered in \cite{FF}.  
\item Doubling metric measure spaces that support a $2$-Poincar\'e inequality with respect to the upper gradient
structure of Heinonen and Koskela (see~\cite{heinonen_koskela_shanmugalingam_tyson_2015,KOSKELA20142437,Koskela2012}).
\item Metric graphs with bounded geometry (see \cite{Haeseler}).
\item Abstract Wiener spaces  (see e.g. \cite{bogachev} for a general introduction to Gaussian spaces ). The extension theorem was introduced in  \cite{NPS1,NPS2}.
\end{itemize}

\begin{remark}
In the paper by Stinga and Torrea \cite{stingatorrea}, the previous theorem was  proved on the Euclidean case with positive measure \cite{Caffarelli_2007,stingatorrea}, Gauss spaces  and some variations of them on bounded domains. 
\end{remark}

\begin{remark}
It is clear that the arguments in \cite{Kwasnicki2018} can be extended to a large class of Dirichlet spaces, in exactly the same way it can be done for pure powers as described above.  
\end{remark}

\subsubsection{Sketch of the argument for Theorem \ref{thm:extension_theorem}}
\label{sec:proof_main}

	1. First, we check the following integral is convergent
\begin{align}
	\left< U(\cdot \ , y), \ g(\cdot)\right>_{L^2(X;\ \mu)} = \frac{1}{\Gamma(s)}\int_0^\infty \left<P_t(-L)^sf, g\right>_{L^2(X;\mu)}  e^{-\frac{y^2}{4t}}\frac{dt}{t^{1-s}},
\end{align}
for almost all $y>0$ and for all $g \in L^2(X;\mu)$. From now on, we will use $L^2(X)$ to denote $L^2(X; \mu)$ for simplicity whenever there is no confusion. For each $R>0$, we can define almost everywhere that
    $$ U_R(x,y) = \frac{1}{\Gamma(s)}\int_0^R (P_t(-L)^sf)(x)  e^{-\frac{y^2}{4t}}\frac{dt}{t^{1-s}}.$$
    Since $f \in \mathcal{D}((-L)^s)$, we have  $P_t(-L)^sf \in L^2(X)$ and hence $U_R$ is well-defined. Moreover,  
    \begin{align*} 
       & \left< U_R(\cdot \ , y), \ g(\cdot)\right>_{L^2(X)} = \frac{1}{\Gamma(s)}\int_0^R \left<P_t(-L)^sf, g\right>_{L^2(X)} e^{-\frac{y^2}{4t}}\frac{dt}{t^{1-s}}\\
        &= \frac{1}{\Gamma(s)}\int_0^\infty \int_0^R e^{-t\lambda} (t\lambda)^s
        e^{-\frac{y^2}{4t}}\frac{dt}{t} dE_{f,g}(\lambda).
    \end{align*}
    We change the order of integration because of the integrability. By the change of variable $r = t\lambda$, we have 
    \begin{align*}
        \left|\left<U_R(\cdot, y), \ g(\cdot)\right>_{L^2(X)} \right| \leq
        \frac{1}{\Gamma(s)}\int_0^\infty \int_0^\infty e^{-r} r^s
        \frac{dr}{r} d\left|E_{f,g}(\lambda)\right|
        \leq \norm{f}_{L^2(X)} \norm{g}_{L^2(X)}.
    \end{align*}
    Therefore, for each fixed $y >0$, $U_R(\cdot, y)$ is in $L^2(X)$ and 
    $$ \norm{U_R(\cdot, y)}_{L^2(X)} \leq \norm{f}_{L^2(X)}.$$
    By the similar computation, for some $R_2 > R_1 > 0$, 
    \begin{align*}
        \abs{\ang{U_{R_1}(\cdot, y), g} - \ang{U_{R_2}(\cdot, y), g}}
        & \leq 
        \frac{1}{\Gamma(s)} \int_0^\infty e^{-r} r^s
        \frac{dr}{r} \int_{R_1}^{R_2} d\left|E_{f,g}(\lambda)\right| \to \ 0
    \end{align*}
    as $R_1, \ R_2 \to \infty$. There exist a Cauchy sequence of bounded operators $\crl{U_{R^j}(\cdot, \ y)}_{j \in \mathbb{N}}$ converging in $L^2(X)$ for almost all $y\in \R$. By the dominated convergence theorem, we can pass the limit and have 
    \begin{align*}
        \left< U(\cdot \ , y), \ g(\cdot)\right>_{L^2(X)} = 
        \lim_{R^j \to \infty} \left< U_{R^j}(\cdot \ , y), \ g(\cdot)\right>_{L^2(X)}
        =
        \frac{1}{\Gamma(s)}\int_0^\infty 
        \left<P_t(-L)^sf, g\right>_{L^2(X)}
        e^{-\frac{y^2}{4t}}\frac{dt}{t^{1-s}}.
    \end{align*}
    2. By a similar limit argument, one can check that $U(\cdot, y)\in \mathcal{D}(L)$ and the boundary condition in \eqref{eq:entension_equation} holds. The partial derivative of $U$ can be passed inside the integral based on the convergence of the integral. The equation \eqref{eq:Possion_formula} and \eqref{eq:D2N} follow from similar computation.

\subsection{Extended Dirichlet space and tensorization} 

In this section, we  see $X_a$  as a Dirichlet space. It will be important for further analysis since, as in \cite{Caffarelli_2007}, the idea is to derive properties of solutions of an equation on  $X_a=X \times \mathbb R$ out of the properties of the Dirichlet space $X$ and {\sl vice versa}. As always, we assume that the Dirichlet form $(\mathcal{E},\mathcal{F})$ is regular and strongly local.  One can consider the Dirichlet form on $\mathbb{R}$ given by
\[
\tilde{\mathcal{E}}(f,f)=\int_\mathbb{R} f'(y)^2 d\nu_a(y), \quad f \in H^1(\mathbb R,\nu_a).
\]
 It is strongly local and regular and we can consider the tensorized Dirichlet form $(\mathcal{E}_a,\mathcal{F}_a)$ on the product space $(X_a, d\nu_a)$, see section \ref{tensorization Dirichlet space}. 

\subsection{Sub-Gaussian estimates and the existence of  weak solutions}\label{sec:existence}

{This section is devoted to various results about existence of solutions. We would like to point out that we will be actually dealign with three different Dirichlet forms:
\begin{itemize}
\item The underlying Dirichlet form $\mE$ with generator $L$.
\item The Dirichlet form $\mE^{(s)}$ with generator $(-L)^s$.
\item The Dirichlet form $\mE_a$ on the product space $X_a$ defined above. 
\end{itemize}
Our main assumption is that the generator $L$ satisfies sub-Gaussian estimates as in \eqref{eq:subGaussian}.}
 To proceed, we first define Sobolev-Besov spaces in terms of the heat semigroup. This approach was undertaken for instance in \cite{alonso2021besov} and in \cite{grigor2003heat}. For $f\in L^2(X, \mu)$ and $r > 0$, denote
\begin{equation}\label{eq:N_seminorm}
	D(f, r) := \iint_{\Delta_r} \abs{f(x) - f(y)}^2 d\mu(x) d\mu(y)
\end{equation}
and for some $\alpha, \beta > 0$, 
\begin{equation}\label{eq:besov_norm}
	N_{\alpha, \beta}(f) := \sup_{r\in(0, 1]}\frac{D(f, r)}{r^{\alpha + \beta}}
\end{equation}
where 
\begin{equation}
	\Delta_r = \crl{(x, y) \in X\times X: d(x,y) < r}.
\end{equation}
Then we define the Besov space $\frak{B}^{\alpha, \beta}_2(X)$  as
\begin{equation}
	\frak{B}^{\alpha, \beta}_2(X):=\crl{f \in L^2(X, \mu) : N_{\alpha, \beta}(f) < \infty}.
\end{equation}

We define the following weak solution to the fractional operator $(-L)^su = 0$ with boundary condition by means of a Dirichlet principle. 
\begin{definition}
	Consider a bounded domain $\Omega \subset X$ such that $\mu(X\backslash \Omega) >0$ and a function $f \in \frak{B}^{d_H, sd_W}_2(X)$. Then we call a function $u\in \frak{B}^{d_H, sd_W}_2(X)$ a weak solution to the Dirichlet problem $(-L)^s u = 0$ on $\Omega$ with boundary data $f$ if $u$ satisfies 
	\begin{enumerate}
		\item $u = f$ almost everywhere on $X \backslash \Omega$, and 
		\item for all $h \in \frak{B}^{d_H, sd_W}_2(X)$ satisfying $h = f$ almost everywhere in $X \backslash \Omega$, we have 
		\begin{equation}\label{eq:weak_sol_frac_L}
			\mE^{(s)} (u, u) \leq \mE^{(s)} (h, h),	
		\end{equation}
		{where $\mE^{(s)}$ denotes the nonlocal Dirichlet form $\mE^{(s)}(h,h)=((-L)^s h, h)$.}

	\end{enumerate}
\end{definition}

We are ready to present the following existence result. 
\begin{theorem}\label{Main_thm:ext}
Assume \eqref{eq:subGaussian}. 
Let $f\in \frak{B}^{d_H, sd_W}_2(X)$ and $\Omega$ be a bounded domain in $X$ with $\mu(X\backslash \Omega) >0$. Then there is a unique $u\in\frak{B}^{d_H, sd_W}_2(X)$ with $u = f \in X\backslash\Omega$ such that whenever $h \in \frak{B}^{d_H, sd_W}_2(X)$ with $h = f $ in $X \backslash\Omega$, we have 
\begin{equation}
	\mE^{(s)}(u, u) \leq \mE^{(s)}(h, h).	
\end{equation}
Equivalently, we have 
\begin{equation}
	\mE^{(s)}(u,h) = 0,	
\end{equation}
whenever $h\in \frak{B}^{d_H, sd_W}_2(X)$ such that $h$ has compact support in $\Omega$. 
\end{theorem}

{
Theorem \ref{Main_thm:ext} follows as in \cite{gianmarco}. Indeed this does not rely that much on the metric structure than the correct definition of Sobolev functions to implement the minimization scheme. The only point to check is that in our framework of sub-Gaussian estimates, the norm on the Besov space $\frak{B}^{\alpha, \beta}_2(X)$ is equivalent to $\mE(f,f)$ for any admissible $f$. This is the object of the following proposition.}

\begin{proposition}\label{prop:fractional_Dirichlet_besov}
	Assume \eqref{eq:subGaussian}. Let $f\in \frak{B}^{d_H, sd_W}_2(X)$. Then the following holds
	\begin{equation}
		\mE^{(s)}(f,f) \simeq N_{d_H, sd_W}(f)
	\end{equation}
\end{proposition}

\begin{proof}
The proof  follows mainly the  reasoning in \cite{grigor2003heat}. We first apply  \cite[Lemma 5.4]{grigor2003heat} and proof of \cite[Corollary 5.5]{grigor2003heat} so that we have an estimation of the heat kernel $q_t(x, y)$ of $(-L)^s$, namely 
\begin{equation}
	q_t(x, y) \simeq\frac{1}{t^{d_H/sd_W}}\Phi\brak{\frac{d(x, y)}{t^{d_H/sd_W}}},
\end{equation}
where $$\Phi(x) = \frac{1}{(1 + x)^{d_H + sd_W}}.$$
Then the Dirichlet form satisfies the condition of \cite[Theorem 5.1]{grigor2003heat}, with $\alpha = d_H$ and $\beta = sd_W$ and the results follows. 
\end{proof}

\begin{proof}[Proof of Theorem \ref{Main_thm:ext}]
We sketch the argument and refer the reader to \cite{gianmarco} for more details. 
 	Let $f \in \frak{B}^{d_H, sd_W}_2(X)$ and assuming that $\mE^{(s)}(f, f)$ is finite and let $\mathcal{K}_f$ denote a subset of $\frak{B}^{d_H, sd_W}_2(X)$ with functions $h$ such that $h = f\  \mu\text{-a.e.} \in X\backslash \Omega$. Let 
 	\begin{equation}
 		I := \inf\{\mE^{(s)}(h, h): h \in \mathcal{K}_f\}
 	\end{equation}
 	If $I = 0$, we have $f$ being a constant on $X \backslash \Omega$; therefore, extending the constant on $X$ yields the desired solution. Without loss of generality, we assume that $I > 0$. From the selection of $f$, $I$ is finite. We choose a minimizing sequence $\{h_k\} \subset \mathcal{K}_f$ with $\mE^{(s)}(h_k, h_k) \leq 2I$. Then we have 
 	\begin{equation}
 		N_{d_H, sd_W}(h_k - f) \lesssim \mE^{(s)}(h_k - f, h_k - f) \leq 6I + \mE^{(s)}(f, f) < \infty. 
 	\end{equation}
 	Since $f \in \frak{B}^{d_H, sd_W}_2(X)$, we have 
 	\begin{equation}\label{eq:N_hk}
 	N_{d_H, sd_W}(h_k) \lesssim \mE^{(s)}(h_k - f, h_k - f) + N_{d_H, sd_W}(f) \leq 6I + \mE^{(s)}(f, f) + N_{d_H, sd_W}(f) := C. 
	\end{equation}
 	
 	Fix any $r\in(0, 1]$, we consider the measure $d\nu_r(x, y) = 1_{\Delta_r}d\mu(x)d\mu(y)$. For each $k$, we denote the function $v_k(x, y) = h_k(x) - h_k(y)$. We see that $\{v_k\}$ is a sequence in $L^2(X\times X, d\nu_r)$ and 
 	$$\norm{v_k}_{L^2(\nu_r)}^2 = D(h_k, r) \leq N_{d_H, sd_W}(h_k)r^{d_H + sd_W} \leq C r^{d_H + sd_W}$$
 	 By taking convex combinations and passing the limit to a subsequence, we have $v_k \to v_\infty$ $\nu_r$-a.e.. If both $x, y \in X\backslash \Omega$, then $v_\infty(x, y) = f(x) - f(y)$. If $x \in  \Omega$ and $y \in X\backslash\Omega$ such that $v_\infty(x, y) = \lim_{k \to \infty}v_k(x, y) = \lim_{k \to \infty}h_k(x) - f(y)$. Here we define the function $h_\infty(x) = v_\infty(x, y) + f(y)$ with $y \in X \backslash \Omega$ as chosen above. Note that by Fubini Theorem, $h-\infty$ is well-defined. Moreover, $v_\infty(x, y)$ also equals to $h_\infty(x) - h_\infty(y)$, therefore 
 	\begin{equation}
 		D(h_\infty, r) = \norm{v_\infty}_{L^2(\nu_r)}^2 \leq C r^{d_H + sd_W}.
 	\end{equation}
 	Hence $N_{d_H, sd_W}(h_\infty) \leq C$ and $h_\infty \in \frak{B}^{d_H, sd_W}_2(X)$. By the lower semicontinuity of $\mE^{(s)}$, 
 	\begin{equation}
 		I \leq \mE^{(s)}(h_\infty, h_\infty)\leq \liminf_{k}\mE^{(s)}(h_k, h_k) = I
 	\end{equation}
 	and $h_\infty$ is the desired solution.
 \end{proof}
 
%
%
%
%
%
%
%
%
%

\section{Regularity of weak solutions and Harnack inequalities}


We emphasize in the following: 

\vspace{0.2cm}
{\sl Throughout the section, we assume that the underlying Dirichlet space $(X, d, \mu, \mE, \mF)$ has the sub-Gaussian heat kernel estimate as in \eqref{eq:subGaussian}.}

\vspace{0.2cm}

It is clear that the  heat kernel $q_t$ in the extended Dirichlet space $(\mE_a, \F_a)$ satisfies the following estimates

%


\begin{lemma}
    The heat kernel $q_t$ associated to the heat semigroup $Q_t$ of the extension space $(X_a, \mE_a, \F_a, \mu_a, d_a)$ satisfies the following estimate (HKE-a).
\begin{align}\label{eq:HKE}
q_t(z, z') \leq \frac{c_1}{\nu_a(B(z_y, \sqrt{t}))t^{d_H/d_W}}
\exp{\brak{-c_2\brak{\frac{d(z_x, z_x')^{d_W}}{t}}^{\frac{1}{d_W-1}} -c_3\frac{|z_y - z_y'|^2}{t}}}, \notag\\
q_t(z, z') \geq \frac{c_4}{\nu_a(B(z_y, \sqrt{t}))t^{d_H/d_W}}
\exp{\brak{-c_5\brak{\frac{d(z_x, z_x')^{d_W}}{t}}^{\frac{1}{d_W-1}} -c_6\frac{|z_y - z_y'|^2}{t}}}.
\end{align}
	for $\mu\times \mu$-a.e. $(z, z')\in X \times X$ and each $t\in(0, +\infty)$, where $c_1, c_2, c_3, c_4 ,c_5, c_6> 0$ are constants independent of $z, \ z'$ and $t$.
\end{lemma}

Notice that this is not a classical sub-Gaussian heat kernel estimate, since the speed of heat propagation (i.e. the space-time scaling) in $X_a = X\times \R$ is different in $X$ and $\R$. In the following discussion, we use the notation $B(x, R)$ with $ x\in X$ to denote a ball in $X$ with radius $R$, and $B(y, R)$ with $y\in \R$ to denote a ball in $\R$ with radius in $R$. Note the ambient space of the ball is determined by the center for simplicity. We denote the product of balls as 
\begin{equation}
    D = D(z_0, R) = B(x_0, R^{2/d_W})\times B(y_0, R)
\end{equation}
where $z_0 = (x_0, y_0)$. 

We first derive a special and inhomogeneous version of Parabolic Harnack inequality for the weak solution $U$, from the heat kernel estimate \eqref{eq:HKE}, for which the underlying domain is substituted to $D$ cross the time interval. Notice that $U$ is an extension of $f$, satisfying $U(\cdot, 0) = f(\cdot)$. Then we take the trace operator to $U$. Since the direction in $\R$ is projected, we are left with the classical elliptic Harnack inequality for $f$. We emphasize that the classical parabolic Harnack inequality for $U$, which is equivalent to the classical sub-Gaussian heat kernel estimate on $X_a$, cannot be achieved, because of the difference of propagation speed in $X$ and $\R$ as we mentioned above.

This section goes as follows. First, we prove that the weak local lower estimate (LLE) holds on $X_a$ (Theorem \ref{thm:HKE_a_LLE}). Next, we prove that (LLE) implies the oscillation inequality (Proposition \ref{prop:OSC}) and an inhomogeneous H\"{o}lder continuity for the caloric functions in $X_a$ (Definition \ref{def:harmonic_functions}, Corollary \ref{cor:Inhomo_holder}). Then we prove the inhomogeneous parabolic Harnack inequality for the caloric functions in $X_a$ (Theorem \ref{thm:LLE_PHI}). At last, we prove the elliptic Harnack inequality and H\"{o}lder continuity for the weak solutions of $(-L)^sf = 0$ (Theorem \ref{thm:EHI_weak_solution_Ls}).

%


\subsection{Weak local lower estimate}
Define the following weak local lower estimate (LLE).
\begin{definition}\label{def:LLE}
We say the extended Dirichlet space $(X_a, d_a, \mu_a, \mE_a, \mF_a)$ satisfies the weak local lower estimate (LLE), if there exists $\varepsilon\in(0, 1)$ such that for all $z_0 = (x_0, y_0)\in X_a$ with $x_0\in X$, $y_0 \in \R$ and $R>0$, there exists a heat kernel $q_t^{D(z_0, R)}$ of the semigroup $\crl{Q_t^{D(z_0, R)}}$ that satisfies the estimate
\begin{equation}\label{eq:LLE}
q_t^{D(z_0, R)}(z, z') \geq \frac{c}{\nu_a(B(y_0, \sqrt{t}))t^{d_H/d_W}}
\end{equation}
for $\mu_a\times \mu_a$-almost $z, z'\in D(z_0, \varepsilon\sqrt{t}) = B(x_0, (\varepsilon\sqrt{t})^{2/d_W})\times B(y_0, \varepsilon \sqrt{t})$  and all $0<t\leq \brak{\varepsilon R}^2$, with some positive constant $c$.
\end{definition}
\begin{remark}
	Notice that by  \eqref{eq:ball_size}, 
	\begin{equation}
		t^{d_H/d_W} \asymp \mu(B(x, t^{1/d_W})).
	\end{equation}
\end{remark}

We now prove that the estimate (HKE-a) of $q_t(z, z')$ in \eqref{eq:HKE} implies the weak local lower estimate (LLE).

\begin{theorem}[HKE-a $\Rightarrow$ LLE]\label{thm:HKE_a_LLE}
    The inequality (HKE-a) \eqref{eq:HKE} on $X_a$ implies that (LLE) \eqref{eq:LLE} holds on $X_a$. 
\end{theorem}
\begin{proof}
First, we prove that the semigroup $Q_t^D$ possesses a heat kernel $q_t(z, z')^D$.  By the heat kernel estimate, 
\begin{equation}
    \esssup_{z, z' \in D}q_t(z, z') \leq \sup_{z \in D}\frac{c_3}{\nu_a(B(z_y, \sqrt{t}))t^{d_H/d_W}} 
\end{equation}
When $0 < \sqrt{t} \leq R$, we have by the volume doubling property of the measure $\nu_a$ and \eqref{eq:VD_gamma}, for all $z \in D$, 
\begin{equation}
     \frac{1}{\nu_a(B(z_y, \sqrt{t})}\leq  \frac{C_{VD}}{\nu_a(B(y_0, R))}\brak{\frac{d(z_y, y_0)+R}{\sqrt{t}}}^\gamma\leq  \frac{C_{VD}}{\nu_a(B(y_0, R))}\brak{\frac{2R}{\sqrt{t}}}^\gamma.
\end{equation}
When $\sqrt{t}>R$, notice that $B(y_0, R) \subset B(z_y, 2\sqrt{t})$. By the volume doubling property directly, 
\begin{equation}
    \frac{1}{\nu_a(B(z_y, \sqrt{t}))} \leq  \frac{C_{VD}}{\nu_a(B(z_y, 2\sqrt{t}))} \leq 
     \frac{C_{VD}}{\nu_a(B(y_0, R))}. 
\end{equation}
Now we have 
\begin{equation}
    \esssup_{z, z' \in D}q_t(z, z') \leq F(t),
\end{equation}
where 
\begin{equation}
	F(t) = \left\{
	\begin{array}{lr}
		\frac{C_{VD}}{\nu_a(B(y_0, R))t^{d_H/d_W}}\brak{\frac{2R}{\sqrt{t}}}^\gamma, \ &0 \leq t \leq R^2\\
		\frac{C_{VD}}{\nu_a(B(y_0, R))t^{d_H/d_W}}, \ &t > R^2
	\end{array}
	\right.
\end{equation}
and $F(t)$ is a function independent of $z$ and $z'$. Therefore, for any non-negative function $f \in L^1(D)$ and $\mu_a$-almost all $z \in D$, 
\begin{equation}
    Q_t^D f(z) \leq Q_t f(z) = \int_{D}q_t(z,z')f(z')d\mu(z') \leq F(t) \norm{f}_{L^1}.
\end{equation}
Hence the semigroup $Q_t^D$ is $L^1 \to L^\infty$ ultracontractive, which implies the existence of the heat kernel $q_t^D(z, z')$. 

For the following discussion, we first fix $t \leq (\varepsilon R)^2$ with $0<\varepsilon < 1/2$ to be specified later. 
By \cite[Lemma 4.18]{grigor2008off}, for any open set $U\subset X_a$ and any compact $K\subset U$, for any non-negative function $f\in L^2(X_a, \mu_a)$ and any $t > 0$, the following holds.

\begin{equation}\label{eq:kernel_global_local}
    Q_tf(z) \leq Q_t^U f(z) + \sup_{s\in[0, t]} \esssup_{K^c}Q_s f.
\end{equation}
for $\mu_a$-almost $z \in X_a$. We take $U = D = D(z_0, R)$, and
\begin{align}
    K &= \overline{D(z_0, R/2)} = \overline{
    B(x_0, (R/2)^{2/d_W})
    \times B\brak{y_0, \frac{R}{2}}}, \\
    A &= D(z_0, \varepsilon \sqrt{t}) = B(x_0, (\varepsilon\sqrt{t})^{2/d_W})\times B\brak{y_0, \varepsilon \sqrt{t}}
\end{align}
Let $f$ be a non-negative function from $L^1(A)$, we have 
\begin{equation}
    \sup_{s\in[0, t]}\esssup_{z \in K^c} Q_s f(z) = \sup_{s\in(0, t]} \esssup_{z \in K^c}\int_A q_s(z, z')f(z')d\mu(z') \leq M \norm{f}_{L^1},
\end{equation}
where 
\begin{equation}
    M := \sup_{s \in(0, t]}\esssup_{z \in K^c, z' \in A} q_s(z, z').
\end{equation}
Notice that the value $s = 0$ can be dropped from $\sup_{s\in[0, t]}$ because $Q_0 f = f$ and $\esssup_{z \in K^c}f(z) = 0$.

Multiplying \eqref{eq:kernel_global_local} by a non-negative function $g \in L^1(A)$ and integrating, we have 
\begin{equation}
    \int_A (Q_tf)g d\mu \leq \int_A (Q_t^Df)g d\mu + M \norm{f}_{L^1} \norm{g}_{L^1}.
\end{equation}
It is equivalent to 
\begin{equation}
    \int_A\int_A q_t(z, z')f(z) g(z') d\mu(z) d\mu(z') \leq 
    \int_A\int_A q_s^D(z, z')f(z)g(z') d\mu(z) d\mu(z') + M \norm{f}_{L^1}\norm{g}_{L^1}.
\end{equation}
Dividing by $\norm{f}_{L^1}\norm{g}_{L^1}$ and taking inf in all test functions $f, g$, we obtain 
\begin{equation}\label{eq:sep}
    \essinf_{z, z'\in A}q_t(z, z') \leq \essinf_{z, z' \in A}q_t^D(z, z') + M.
\end{equation}

%

By the definition of (LLE) as in Definition \ref{def:LLE}, we need to estimate $\essinf_{z, z'\in A}q_t(z, z')$ from below and $M$ from above, with $\varepsilon$ to be chosen later. 

%

%

For the first estimation, we refer to the lower heat kernel bound as in \eqref{eq:HKE}. 
\begin{equation}
    \essinf_{z, z' \in A} q_z(z, z')\geq  \frac{c_3}{\nu_a(B(z_y, \sqrt{t}))t^{d_H/d_W}}
\exp{\brak{-c_4\brak{\frac{d(z_x, z_x')^{d_W}}{t}}^{\frac{1}{d_W-1}} -c_5\frac{|z_y - z_y'|^2}{t}}}
\end{equation}
from the definition of $A$, we have 
\begin{equation}
    d(z_x, z'_x)^{d_W} \leq  \brak{d(z_x, x_0) + d(z'_x, x_0)}^{d_W} \leq (2^{d_W}\varepsilon^2)
    t \leq t, 
\end{equation}
\begin{equation}
    \abs{z_y - z'_y}^2 \leq  \brak{\abs{z_y - y_0} + \abs{z'_y - y_0}}^2 \leq \brak{2\varepsilon}^2t \leq t,
\end{equation}
provided $\varepsilon \leq 2^{-d_W/2}$. Hence we have the lower bound, 
\begin{equation}\label{eq:low}
    \essinf_{z, z' \in A} q_t(z, z') \geq \frac{ c}{\nu_a(B(z_y, \sqrt{t}))t^{d_H/d_W}}
\end{equation}
for some constant $ c$ depends on the constants in \eqref{eq:HKE}.

For the estimation of $M$, we take $z \in K^c$ and $z' \in A$. Since we assume that $t \leq (\varepsilon R)^2$, we have 
\begin{align*}
    d(z_x, z'_x)^{d_W} \geq \brak{d(z_x, x_0) - d(z'_x, x_0)}^{d_W} \geq \brak{(R/2)^{2/d_W} - (\varepsilon \sqrt{t})^{2/d_W}}^{d_W}\\
    \geq \brak{(1/2)^{2/d_W} - (\varepsilon^2)^{2/d_W}}^{d_W} R^2 \geq \tilde c_1 R^2.
\end{align*}
Such $\tilde c_1$ exists since we assume that $\varepsilon \leq 2^{-d_W/2}$. Similarly, there exists a constant $\tilde c_2$ such that, 
\begin{equation}
    \abs{z_y - z'_y}^2 \geq \brak{\abs{z_y - y_0} - \abs{z'_y - y_0}}^2 \geq \brak{R/2- \varepsilon\sqrt{t}}^2 \geq \brak{1/2 - \varepsilon^2}^2R^2 \geq  \tilde c_2 R^2, 
\end{equation}
We take $\tilde c = \min(\tilde c_1, \tilde c_2)$. 
Then for all $0 < s \leq t$ and $z \in K^c$, $z' \in A$, we have 
\begin{align}
    q_s(z, z') & 
    \leq
    \frac{c_3}{\nu_a(B(z'_y, \sqrt{s}))s^{d_H/d_W}}
\exp{\brak{-c_4\brak{\frac{
\tilde c R^2
}{s}}^{\frac{1}{d_W-1}} 
-c_5 
\frac{\tilde c R^2}{s}}
}\\
& \leq
\frac{c_3}{\nu_a(B(y_0, \sqrt{t}))t^{d_H/d_W}}    \frac{\nu_a(B(y_0, \sqrt{t}))t^{d_H/d_W}}
{\nu_a(B(z'_y, \sqrt{s}))s^{d_H/d_W}}
\exp{\brak{-c_4\brak{\frac{
\tilde c R^2
}{s}}^{\frac{1}{d_W-1}} 
-c_5 
\frac{\tilde c R^2}{s}}
}\\
& \leq
\frac{c_3}{\nu_a(B(y_0, \sqrt{t}))t^{d_H/d_W}}\brak{\frac{\abs{z'_y - y_0} + \sqrt{t}}{\sqrt{s}}}^\gamma
\frac{t^{d_H/d_W}}{s^{d_H/d_W}}\exp{\brak{-c_4\brak{\frac{
\tilde c R^2
}{s}}^{\frac{1}{d_W-1}} 
-c_5 
\frac{\tilde c R^2}{s}}
}\\
& \leq \frac{c_3}{\nu_a(B(y_0, \sqrt{t}))t^{d_H/d_W}}\brak{\frac{R}{\sqrt{s}}}^\gamma
\exp{
-c_5 
\frac{\tilde c R^2}{s}
},
\end{align}
where the third inequality follows from \eqref{eq:VD_gamma} and the last one is because $\abs{z'_y - y_0} + \sqrt{t} \leq (\varepsilon^2 + \varepsilon)R \leq R$ since we assume $\varepsilon \leq 1/2$.

Note that $0 < s \leq t$ and $0 < t < (\varepsilon R)^2$, hence 
\begin{equation}
	\frac{R}{\sqrt{s}}\geq \varepsilon^{-1}.
\end{equation}
Using the fact that, for positive a, b, and c, 
\begin{equation}
	\xi^a \exp(-c\,\xi^b) \to 0 \text{ as } \xi \to \infty. 
\end{equation}
we can conclude that if $\varepsilon$ is small enough, then 
\begin{equation}
	q_s(z, z') \leq \frac{c/2}{\nu_a(B(y_0, \sqrt{t}))t^{d_H/d_W}}
\end{equation}
where $c$ is the constant as in \eqref{eq:low}. Then combining  \eqref{eq:sep} and \eqref{eq:low}, we get 
\begin{equation}
	\essinf_{z, z' \in A}q_t^D(z, z') \geq \frac{\tilde c/2}{\nu_a(B(y_0, \sqrt{t}))t^{d_H/d_W}}
\end{equation}
\end{proof}

\subsection{Oscillation inequality and H\"{o}lder continuity}
We notice that (LLE) implies 
\begin{equation}\label{eq:LLE_other}
	q_t^{D(z_0, R)}(z, z') \geq \frac{c}{\nu_a(B(y_0, R))R^{2d_H/d_W}}
,
\end{equation}
for $\mu_a \times \mu_a$-almost all $z, z' \in D(z_0,  \varepsilon r) = B(x_0, (\varepsilon r)^{2/d_W})\times B\brak{y_0, \varepsilon r}$ provided $r$ and $t$ satisfies the conditions 
\begin{equation}
	r \leq \sqrt{t} \leq \varepsilon R
	\end{equation}
Indeed, when $r \leq \sqrt{t}$, 
$D(z_0, \varepsilon r) \subset D(z_0, \varepsilon\sqrt{t})$, and also $ \sqrt{t} \leq \varepsilon R \leq R$, hence \eqref{eq:LLE_other} holds.

\subsubsection{Oscillation inequality}
For all $s \in \R$ and $z \in X_a$, define the cylinder 
\begin{equation}
	\mC((s, z), r):= (s- r^2, s) \times D(z, r)
\end{equation}
For any set $A \subset \R \times X_a$ and a function $f$ on $A$, define 
\begin{equation}
	\esssup_{A} f = \sup_{t} \esssup_{x: (x,t)\in A} f(t, x),
\end{equation}
and define $\essinf_A f$ analogously. Define 
\begin{equation}
	\osc_A f:= \esssup_A f - \essinf_A f.
\end{equation}

\begin{proposition}\label{prop:OSC}
For any bounded caloric function $u$ in a cylinder $\mC((s, z_0), R)$, the following inequality holds 
\begin{equation}\label{eq:osc_0}
	\osc_{\mC((s, z_0), \delta R)} u \leq \theta \osc_{\mC((s, z_0), R)} u
\end{equation}	
with constants $\delta, \theta \in(0, 1)$ that depends only on the constants in the hypothesis.
\end{proposition}

\begin{proof}
	Let $m(R)$ and $M(R)$ denote the essential infimum and essential supremum of $u$ on $\mC((s, z_0), R)$ respectively. Since $u+$ const is also a caloric function, by the super mean value inequality as in Corollary \ref{cor:super_mean_value}, we have 
	\begin{equation}\label{eq:osc_1}
		u(t,w) - m(R) \geq \int_{D(z_0, R)} q_{t-\xi}^{D(z_0, R)}(w, z) (u(\xi, z) - m(R))\mu_a(dz)
	\end{equation}
	for all $ s-R^2 < \xi < t < s$ and $\mu$-almost all $w\in D(z_0, R)$.

	Taking $\xi = s - (\varepsilon R)^2$, for any $t \in (s-\frac{(\varepsilon R)^2}{2}, s)$, we have $\sqrt{t-\xi} \in (\frac{\varepsilon R}{\sqrt{2}}, \varepsilon R )$. Following from \eqref{eq:LLE_other}, taking $r = \frac{\varepsilon R}{\sqrt{2}}$, we have 
	\begin{equation}\label{eq:osc_2}
		q_{t-\xi}^{D(z_0, R)}(w, z) \geq \frac{c}{\nu_a(B(y_0, R))R^{2d_H/d_W}} \text{ for all $\mu_a$-a.a.}  \ w, z \in D(z_0, \varepsilon r).
	\end{equation}
		
	Taking $\delta = \frac{\varepsilon^2}{\sqrt{2}}$, notice that the set $\mC((s, z_0), \delta R)$ satisfies the condition for both \eqref{eq:osc_1} and \eqref{eq:osc_2}. We could restrict the integration in \eqref{eq:osc_1} to $D(z_0, \delta R)$, use equation \eqref{eq:osc_2} and take the essential infimum in $(t, w) \in D(z_0, \delta R)$, and then
	\begin{equation}
		m(\delta R) - m(R) \geq \frac{c}{\nu_a(B(y_0, R))R^{2d_H/d_W}} \int_{D(z_0, \delta R)} (u(\xi, z) - m(R))\mu_a(dz).
	\end{equation}
	Similarly, for the maximum, 
	\begin{equation}
		M(R) - M(\delta R) \geq \frac{c}{\nu_a(B(y_0, R))R^{2d_H/d_W}} \int_{D(z_0, \delta R)} (M(R) - u(\xi, z))\mu_a(dz).
	\end{equation}
	Taking the difference between the above two equations, because of the volume doubling property, we could find a small enough $c_2<1$ such that, 
	\begin{align*}
		M(R) - m(R) - (M(\delta R) - m(\delta R)) &\geq c\,\frac{\mu_a(D(z_0, \delta R))}{\nu_a(B(y_0, R))R^{2d_H/d_W}}(M(R) - m(R))\\
		&\geq c (M(R) - m(R))\frac{\nu_a(B(y_0, \delta R))(\delta R)^{2d_H/d_W}}{\nu_a(B(y_0, R))R^{2d_H/d_W}}\\
		&\geq c_2 (M(R) - m(R)),
	\end{align*}
	where the second inequality follows from \eqref{eq:ball_size} and the last one from the doubling property of $\nu_a$. Take $\theta = 1-c_2$, we have 
	\begin{equation}
		\theta(M(R) - m(R)) \geq M(\delta R) - m(\delta R),
	\end{equation}
	which proves \eqref{eq:osc_0}.
\end{proof}

\subsubsection{Inhomogeneous H\"{o}lder continuity}
From the oscillation inequality, we could prove the H\"{o}lder continuity of caloric functions. 
\begin{corollary}\label{cor:Inhomo_holder}
	For any bounded caloric function $u$ in a cylinder $\mC((t_0, z_0), R)$, the following inequality holds
	\begin{equation}
		\abs{u(s', z') - u(s'', z'')} \leq C \left(\frac{\sqrt{s' - s''} + d(z'_x,z_x'')^{2/d_W} + \abs{z_y' - z_y''}}{R}\right)^\alpha\osc_{\mC((t_0, z_0), R)}u,
	\end{equation}
	for $dt\times\mu_a$ almost all $(s', z'), (s'', z'') \in \mC ((t_0, z_0), \delta R)$, where $\alpha, \delta \in (0, 1)$ and $C >0$ are constants only depend on the constants in the volume doubling property of $\mu$  and (LLE). 
\end{corollary}

\begin{proof}
	We prove the following equivalent form: for any $r>0$ and $dt\times \mu_a$-almost all $(s', z'), (s'', z'') \in \mC ((t_0, z_0), \delta R)$ such that 
	\begin{align}\label{eq:holder_r}
			\sqrt{s' - s''} + d(z'_x,z_x'')^{2/d_W} + \abs{z_y' - z_y''} < r
	\end{align}
	the following inequality holds:
	\begin{align}
		\abs{u(s', z') - u(s'', z'')} \leq C \brak{\frac{r}{R}}^\alpha \osc_{\mC((t_0, z_0), R)}u.
	\end{align}
	It is sufficient to show that any two points $(s', z'), (s'', z'') \in \mC ((t_0, z_0), \delta R)$ with condition \eqref{eq:holder_r} are contained in an open subset $\Omega \subset \mC((t_0, z_0), R)$ such that 
	\begin{align}
		\osc_{\Omega}u \leq C \brak{\frac{r}{R}}^\alpha \osc_{\mC((t_0, z_0), R)}u.
	\end{align}

	Now we construct the set $\Omega$ as follows. First, we assume that the equation \eqref{eq:holder_r} holds. Suppose $s' \geq s''$. Let $w = z'$ and we wish to choose $t$ so that
	\begin{align}\label{eq:tw}
		\sqrt{t - s} + d(w_x,z_x)^{2/d_W} + \abs{w_y - z_y} < r
	\end{align}
	for both $(s, z) = (s', z')$ and $(s, z) = (s'', z'')$. For $(s, z) = (s', z')$, 
	\begin{equation}
		\sqrt{t - s'} + d(w_x, z_x')^{2/d_W} + \abs{w_y - z_y'} = \sqrt{t - s'}
	\end{equation}
	And for $(s, z) = (s'', z'')$, 
	\begin{equation}
		\sqrt{t - s''} + d(w_x, z_x'')^{2/d_W} + \abs{w_y - z_y''} = \sqrt{t - s''} + d(z_x', z_x'')^{2/d_W} + \abs{z_y' - z_y''}
	\end{equation}
	Because of the strict inequality in \eqref{eq:holder_r}, we can choose $t$ strictly larger than $s'$, so that $s'' \leq s' < t < t_0$. Through this construction, for both $(s, z) = (s', z')$ and $(s, z) = (s'', z'')$, 
	\begin{align}
		\sqrt{t - s} < r &\Longrightarrow s \in (t-r^2, t),\\
		d(w_x, z_x)^{2/d_W} < r &\Longrightarrow z_x \in B(w_x, r^{2/d_W}),\\
		\abs{w_y, z_y} \leq r & \Longrightarrow z_y \in B(w_y, r). 
	\end{align}
	Hence both $(s', z')$ and $(s'', z'')$ are in $\mC((t, w), r)$ once they satisfies \eqref{eq:holder_r}.
	Since $(s', z'), (s'', z'') \in \mC ((t_0, z_0), \delta R)$, we also have $(s', z'), (s'', z'') \in \mC ((t_0, z_0), R)$. Then we can define the set 
	$$\Omega = \mC((t, w), r)\cap \mC((t_0, z_0), R).$$ 
	Note that by construction, we have $(t, w) \in C((t_0, z_0), \delta R)$, i.e.
	\begin{align}\label{eq:tw_0}
		t_0-(\delta R)^2 < t < t_0, \ d(w_x, x_0) \leq (\delta R)^{2/d_W},
		\text{ and } \abs{w_y - y_0}\leq \delta R .
	\end{align}

Next, we find $\delta$, such that $\delta^{-k} r \leq \delta R$ implies $C((t, w), \delta^{-k} r) \subset C((t_0, z_0), R)$ for any integer $k \geq 1$. This condition means
	\begin{equation}
		t_0 - R^2 < t - (\delta^{-k}r)^2, \ 
		d(w_x, x_0) + (\delta^{-k}r)^{2/d_W} \leq R^{2/d_W},
		 \text{ and } \abs{w_y - y_0} + \delta^{-k}r \leq R
	\end{equation}
	Because of \eqref{eq:tw_0}, it suffices to require that 
	\begin{equation}
		(\delta R)^2 + (\delta^{-k} r)^2 \leq R^2,\ 
		(\delta^{-k}r)^{2/d_W} + (\delta R)^{2/d_W} \leq R^{2/d_W}
		\text{, and } \delta R + \delta^{-k} r \leq R.
	\end{equation}
	If we require the above equation hold when $\delta^{-k} r \leq \delta R$, we need to have 
	\begin{equation}
		2(\delta R)^2 \leq R^2, \ 2(\delta R)^{2/d_W} \leq R^{2/d_W}
		 \text{, and } 2\delta R  \leq R.
	\end{equation}
	which requires $\delta \leq \frac{1}{\sqrt{2}}$. 
	
	Given this $\delta$, we first consider the case that $r$ is small, such that $\delta^{-k} r \leq \delta R$ for some integer $k \geq 1$. By the previous argument, we have  $C((t, w), \delta^{-k} r) \subset C((t_0, z_0), R)$ for some $k\geq 1$, we have 
	$$\Omega = C((t, w), r)\subset C((t, w), \delta^{-k} r) \subset  C((t_0, z_0), R).$$
	Hence by oscillation inequality \eqref{eq:osc_0}, 
	\begin{equation}
		\osc_{\Omega} u \leq \theta^k \osc_{C((t, w), \delta^{-k}r)}u \leq \theta^k \osc_{C((t_0, z_0), R)}u.
	\end{equation}
	Note that $\delta$ might not be the same constant as in \eqref{eq:osc_0}. But the oscillation inequality still holds by taking small enough $\delta$. 
	Now $\delta^{-k} r \leq \delta R$ implies
	\begin{equation}
		k \geq \lfloor\frac{\log\frac{r}{R}}{\log \delta}\rfloor -1 \geq \frac{\log\frac{r}{R}}{\log \delta } - 2.
	\end{equation}
	Hence $$\theta^k \leq e^{\frac{\log\frac{r}{R}}{\log \delta}\log \theta} \theta^{-2} = C \left(\frac{r}{R}\right)^\alpha,$$
	where $C = \theta^{-2}$ and $\delta < \theta$ such that $\alpha = \frac{\log \theta}{\log \delta} < 1$ (again by taking small enough $\delta$).
	And then 
	\begin{align}
		\osc_{\Omega} u	\leq  C \left(\frac{r}{R}\right)^\alpha \osc_{C((t_0, z_0), R)}u
	\end{align}

	Then we consider large $r$ such that $r >\delta ^2 R$. Notice that $\delta \leq \frac{1}{\sqrt{2}}$ and $\Omega \subset C((t_0, z_0), R)$, we have 
	\begin{equation}
		\osc_{\Omega} u \leq \osc_{C((t_0, z_0), R)} u = C \left(\frac{r}{R}\right)^\alpha \osc_{C((t_0, z_0), R)}u
	\end{equation}
	for any constant $C \geq \left(\frac{R}{r}\right)^\alpha > \delta^{-2\alpha}$. By taking the constant $C = \max(\delta^{-2\alpha}, \theta^{-2})$, we finished the proof.
\end{proof}

\subsection{Proof of the inhomogeneous Harnack inequality}
We will prove inhomogeneous Parabolic Harnack inequality.
\begin{theorem}[Inhomogeneous Parabolic Harnack inequality]\label{thm:LLE_PHI}
	Let $u$ be a bounded non-negative caloric function in the cylinder 
	$$
	Q = (0, (\varepsilon R)^2) \times D(z_0, R) = (0, (\varepsilon R)^2) \times B(x_0, R^{2/d_W})\times B(y_0, R)
	$$	
	where $z_0 = (x_0, y_0)$ and $z_0 \in X, y_0 \in \R$, with arbitrary $R > 0$. Here $\varepsilon$ is the parameter from Definition \ref{def:LLE}, $l = \frac{1}{\sqrt{2}}$, and then 
	$$
	\inf_{Q_+}u \leq 1 \Longrightarrow \sup_{Q_-}u \leq C
	$$ 
	where 
	\begin{align}
		Q_- &= ((l^3\varepsilon R)^2, (l^2\varepsilon R)^2) \times D(z_0, \eta R) = ((l^3\varepsilon R)^2, (l^2\varepsilon R)^2) \times B(x_0, (\eta R)^{2/d_w})\times B(y_0, \eta R),\\
		Q_+ &= ((l\varepsilon R)^2, (\varepsilon R)^2) \times D(z_0, \eta R) = ((l\varepsilon R)^2, (\varepsilon R)^2)\times B(x_0, (\eta R)^{2/d_w})\times B(y_0, \eta R),
	\end{align}
	and the constants $\eta \in (0, 1), C>1$ depend only on the constants in the hypothesis. 
\end{theorem}
The essential part of the proof is the following lemma. 
\begin{lemma}\label{lem:LLE_PHI}
	There exists a constant $\sigma$ with $\eta < \sigma < 1$, $K >1$, and $C>0$, depending only on the constants in the hypothesis, such that the following holds. 
	For every $(t, z) \in Q$ and $ r > 0$ be such that
	\begin{equation}\label{eq:Lem_LLE_PHI_C}
		\mC((t, z), r) \subset \widetilde{Q} := (0, (l^2\varepsilon R)^2) \times D(z_0, \sigma R).
	\end{equation}	
	Suppose $\inf_{Q_+}u \leq  1$ and $u(t, z) \geq \lambda$ with
	\begin{equation}\label{eq:lambda_PHI}
		 \lambda \geq C\left(\frac{R}{r}\right)^{\gamma},
	\end{equation}
	where $\gamma$ is the constant in \eqref{eq:VD_gamma}, then
	$$
	\sup_{\mC((t, z), r) }u \geq K\lambda.
	$$
\end{lemma}
\begin{proof}
By the super-mean-value inequality \eqref{eq:SMV_2}, for all $t < T < (\varepsilon R)^2$ and $\mu_a$-almost all $w \in D(z_0, R)$, we have 
\begin{equation}\label{eq:SMV_Harnack}
	u(T, w) \geq \int_{D(z_0, R)}q_{T-t}^{D(z_0, R)}(w, w') u(t, w') \mu(dw').
\end{equation}
Restrict $T$ to the interval $(l \varepsilon R)^2 < T < (\varepsilon R)^2$. For any $0 < t < (l^2\varepsilon R)^2$, it follows that 

\begin{equation}
	(\kappa R)^2 < T-t < (\varepsilon R)^2,
\end{equation}
where $\kappa = l^2\varepsilon$, which is true since $1 = 1/\sqrt{2}$. Applying the equivalent form of the LLE as in \eqref{eq:LLE_other} with $r = \kappa R$, we obtain 
\begin{equation}\label{eq:LLE_other_Harnack}
	q_{T-t}^{D(z_0, R)}(w, w') \geq \frac{c_1}{\nu_a(B(y_0, R))R^{2d_H/d_W}}
\end{equation}
for almost all $w, w' \in D(z_0, \varepsilon\kappa R)$. 
Assume that $\sigma$ is so small, such that 
\begin{equation}
	\sigma \leq \varepsilon\kappa,
\end{equation}
then \eqref{eq:LLE_other_Harnack} holds for almost all $w, w' \in D(z_0, \sigma R)$. 

Then notice that $D(z, r) \subset D(z_0, \sigma R) \subset D(z_0, R)$ by the assumption $\mC((t, z), r) \subset \tilde Q$, we restrict the integration of \eqref{eq:SMV_Harnack} to $D(z, r)$, so that 
\begin{align}
	u(T, w) & \geq \int_{D(z, \delta r)}q_{T-t}^{D(z_0, R)}(w, w') u(t, w') \mu_a(dw')\\
	& \geq \frac{c_1}{\nu_a(B(y_0, R))R^{2d_H/d_W}}
	\int_{D(z, \delta r)} u(t, w') \mu_a(dw')
\end{align}
where {$\delta \in (0,1)$ is as in Proposition \ref{prop:OSC}} and  in the second inequality we applied \eqref{eq:LLE_other_Harnack}. In particular, this inequality holds for all $(T, z) \in Q_+$ since we assumed $\eta < \sigma$. Taking the infimum in $(T, z) \in Q_+$ and using that $\inf_{Q_+} u \leq 	1$, we obtain, 
	\begin{equation}
		1 \geq \frac{c_1}{\nu_a(B(y_0, R))R^{2d_H/d_W}}
	\int_{D(z, \delta r)} u(t, w') \mu_a(dw')
	\end{equation}
	And using  \eqref{eq:ball_size}, we find
	\begin{equation}
		\inf_{w' \in D(z, \delta r)}u(t, w') \leq \Lambda := \frac{\nu_a(B(y_0, R))R^{2d_H/d_W}}{c_1\nu_a(B(z_y, \delta r))(\delta r)^{2d_H/d_W}},
	\end{equation}
	for all $t \in (0, (l^2\varepsilon R)^2)$. Combining with the hypothesis that $u(t, z) \geq \lambda$, we see that 
	\begin{equation*}
		\osc_{\mC((t, z), \delta r)}u \geq \lambda - \Lambda,
	\end{equation*}
	Whence by Proposition \ref{prop:OSC}, 
	\begin{equation*}
		\sup_{\mC((t, z), r)} u \geq \osc_{\mC((t, z), r)} u \geq \theta^{-1} (\lambda - \Lambda).
	\end{equation*}
	where $\theta \in (0, 1)$ is the constant from Proposition \ref{prop:OSC}. We are left to check that
	\begin{align}
		\theta^{-1}(\lambda - \Lambda) \geq K\lambda,
	\end{align}
	with a constant $K > 1$. 
	By the \eqref{eq:VD_gamma}, we have 
	\begin{align}
		\Lambda \leq \frac{C_{VD}}{c_1} \brak{\frac{\abs{y_0 - z_y}+ R}{r}}^\gamma \brak{\frac{R}{\delta r}}^{2d_H/d_W}
		\leq C_1 \brak{\frac{R}{r}}^{\gamma + 2d_H/d_W},
	\end{align}
	where $C_1 = C_1(c_1, C_{VD}, \gamma, \delta)$. If we let 
	\begin{align}
		\lambda \geq \frac{2C_1}{1-\theta}\brak{\frac{R}{r}}^{\gamma + 2d_H/d_W},
	\end{align}
	Then we have $\Lambda \leq \lambda((1-\theta)/2)$, hence
	\begin{align}
		\theta^{-1} (\lambda - \Lambda) \geq \frac{\theta^{-1} +1}{2}\lambda,
	\end{align}
	so that we can let $K = (\theta^{-1} +1)/2>1$ and this completes the proof.
\end{proof}

Given this lemma, we can prove the Theorem \ref{thm:LLE_PHI}. 
\begin{proof}[Proof of Theorem \ref{thm:LLE_PHI}]
	First, we can derive the following property from Lemma \ref{lem:LLE_PHI}.  Define a function 
	\begin{equation}
		\rho(\lambda) = \frac{R}{(C^{-1} \lambda)^{1/\gamma}}
	\end{equation}
	so that the condition \eqref{eq:lambda_PHI} is equivalent to $r \geq \rho(\lambda)$. Take $r = \rho(\lambda)$ as in \eqref{eq:Lem_LLE_PHI_C}. If for some point $(s, w) \in Q$, we have $\lambda := u(s, w) > 0$ and 
	\begin{equation}
		\mC((s, w), \rho(\lambda)) \subset \widetilde{Q},	
	\end{equation}
	then there exists a point $(s', w') \in \mC((s, w), \rho(\lambda))$ such that $u(s', w') \geq K \lambda$. 
	
	We can start with an arbitrary point $(s_0, w_0) \in Q_-$ where $\lambda_0 = u(s_0, w_0) > 0$. Assuming that 
	\begin{equation}
		\mC((s_0, w_0), \rho(\lambda_0)) \subset \widetilde{Q},	
	\end{equation}
	we can choose a point $(s_1, w_1) \in \mC((s_0, w_0), \rho(\lambda_0))$ where 
	\begin{equation}
		\lambda_1 = u(s_1, w_1)\geq K \lambda_0.
	\end{equation}
	If again
	\begin{equation}
		\mC((s_1, w_1), \rho(\lambda_1)) \subset \widetilde{Q},	
	\end{equation}
	we can also find a point $(s_2, w_2) \in \mC((s_1, w_1), \rho(\lambda_1))$ such that 
	\begin{equation}
		\lambda_2 = u(s_2, w_2)\geq K \lambda_1 \geq K^2 \lambda_0.
	\end{equation}
	Following this procedure, we obtain a sequence of points $\{(s_n, w_n)\}$ such that 
	\begin{equation}
		\lambda_n = u(s_n, w_n) \geq K^n \lambda_0.
	\end{equation}
	and 
	\begin{equation}
		(s_n, w_n) \in C((s_{n-1}, w_{n-1}), \rho(\lambda_{n-1})) \subset \widetilde{Q}.
	\end{equation}
	
	We will continue the construction until 
	\begin{equation}\label{eq:finite_n}
		C((s_n, w_n), \rho(\lambda_n)) \not\subset \widehat{Q}:= [(l^4\varepsilon R)^2, (l^2\varepsilon R)^2]\times\overline{D(z_0, \sigma R)}
	\end{equation}
	If such $n$ does not exit, then we obtain an infinite sequence $(s_n, w_n) \in \widehat{Q}$ such that $u(s_n, w_n)\rightarrow \infty$, which is not possible since the function $u$ is bounded in $\widehat{Q} \subset Q$.  
	
	Hence there exists an $n$ such satisfies \eqref{eq:finite_n}, which implies 3 cases, 
	\begin{equation}
		w^x_n \notin B(x_0, (\sigma R)^{2/d_W}),\ 
		w^y_n \notin B(y_0, \sigma R), \text{ or }
		s_n < l^4 \varepsilon R,
	\end{equation}
	where we use the notation $w_n = (w^x_n, w^y_n)$ with $w^x_n \in X$ and $w^y_n \in \R$. 
	
	In the first case, we have 
	\begin{equation}
		 d(w^x_0, w^x_n)\geq d(x_0, w^x_n) - d(x_0, w^x_0) \geq (\sigma^{2/d_W} - \eta^{2/d_W})R^{2/d_W}
	\end{equation}
	On the other hand, 
	\begin{equation}
		 d(w^x_0, w^x_n)\leq \sum_{k = 0}^{n-1}d(w_k^x, w^x_{k+1})\leq \sum_{k = 0}^{n-1}\rho(\lambda_k)^{2/d_W}\leq \sum_{k = 0}^{n-1}\rho(K^k \lambda_0)^{2/d_W} \leq C_2 R^{2/d_W}\lambda_0^{-2/(\gamma d_W)}
	\end{equation}
	Similarly, in the second case, we have the lower bound for $\abs{w^y_0- w^y_n}$,
	\begin{equation}
		 \abs{w^y_0-w^y_n}\geq \abs{y_0-w^y_n} - \abs{y_0-w^y_0} \geq (\sigma - \eta)R
	\end{equation}
	and then the upper bound,
	\begin{equation}
		 \abs{w^y_0-w^y_n}\leq \sum_{k = 0}^{n-1}\abs{w_k^y - w^y_{k+1}}\leq \sum_{k = 0}^{n-1}\rho(\lambda_k)\leq \sum_{k = 0}^{n-1}\rho(K^k \lambda_0) \leq C_3 R\lambda_0^{-1/\gamma }.
	\end{equation}
	In the third case, we have the lower bound for $s_0 - s_n$, 
	\begin{equation}
		s_0 - s_n \geq (l^3\varepsilon R)^2 - (l^4 \varepsilon R)^2 \geq (l^4 \varepsilon R)^2
	\end{equation}
	and then the lower bound,
	\begin{equation}
		s_0 - s_n = \sum_{k = 1}^{n-1} (s_k - s_{k +1}) \leq \sum_{k = 1}^{n - 1}\rho(\lambda_k)^2 \leq C_4 R^2 \lambda_0^{-2/\gamma}.
	\end{equation}
	Comparing the lower bound and the upper bound of $d(w_0^x, w_n^x)$, $\abs{w_0^y - w_n^y}$, and $s_0 - s_n$ in all three cases, we obtain that 
	\begin{equation}
		\lambda_0 \leq C_5 = C_5(C_2, C_3, C_4, \sigma, \eta, \gamma).
	\end{equation}
	Since $\lambda_0$ is the value of $u$ at an arbitrary point in $Q_-$, it follows that $sup_{Q_-} u \leq C_4$, which finished the proof of (PHI).
\end{proof}

\subsection{Proof of Theorem \ref{thm:EHI_weak_solution_Ls}.}
We are now ready to provide the proof of our main result Theorem \ref{thm:EHI_weak_solution_Ls}. 
Since (HKE) holds on $(X, d, \mu, \mE, \mF)$, we have by Lemma \ref{lem:LLE_PHI}, (HKE-a) holds on $(X_a, d_a, \mu_a, \mE_a, \mF_a)$. Then following from Theorem \ref{thm:HKE_a_LLE}, we see that (LLE) holds on $(X_a, d_a, \mu_a, \mE_a, \mF_a)$. At last, we have the Parabolic Harnack inequality for non-negative caloric functions by Theorem \ref{thm:LLE_PHI}.

Suppose $f$ is a non-negative weak solution of $(-L)^sf = 0$. By Theorem \ref{thm:extension_theorem}, there exists $U$ such that $U(\cdot, 0) = f$ and, since by construction the conormal derivative $$- \frac{2^{2s-1} \Gamma(s)}{\Gamma(1-s)} \underset{y\to 0^+}{\lim} y^a \frac{\partial U}{\partial y}(\cdot,y)$$ vanishes, one can extend evenly $U$ so that it is harmonic in $X_a$ (see \cite{Caffarelli_2007}). By taking $\tilde U(t, x, y) = U(x, y)$, which is a constant in the time variable $t$, it is clear that $\tilde U$ is a caloric function. Hence the inhomogeneous parabolic Harnack inequality and the inhomogeneous H\"{o}lder continuity hold for $\tilde U$. Using the fact that $\tilde U$ is independent of $t$, we take the trace operator and project the $y$ direction of $\tilde{U}$, and then derive the Harnack inequality for $f$. The H\"{o}lder continuity of $f$ follows similarly from the projection.

\subsection{Some remarks on {\sl boundary} Harnack principles}

In this Section, we briefly discuss boundary Harnack principles (BHP) for some generators on Dirichlet spaces, under natural but more restrictive assumptions than the ones of the previous sections. In the rest of the section, we always assume that the Dirichlet space under consideration admits {\sl an interior Harnack inequality}. Boundary Harnack inequalities are more subtle properties than their interior versions. They require some sort of regularity and/or geometry of the boundary. When dealing with powers of generators of local Dirichlet forms (namely our setup), the situation starts becoming even more involved since one has two routes (which do not seem to be equivalent in our general setting) to prove such results: 

\begin{itemize}
\item One can start by extending the $s$-harmonic function from $X$ into $X_a$ and then apply a version of Boundary Harnack on the space $X_a$ for the new operator. We then come back to the original space $X$ taking a trace. 
\item One derives directly a BHP on $X_a$ and then taking the trace one defines first a solution of the original harmonic function as a trace of the extended function and as a by-product one gets the desired BHP on $X$.  
\end{itemize}  

The two previous routes are equivalent in some geometries, like for instance the Euclidean case for Lipschitz domains \cite{Caffarelli_2007}. The argument in \cite{Caffarelli_2007} can be easily extended to other nice geometric contexts under natural assumptions. However, it is known \cite{BKT} that powers of the Euclidean laplacian satisfy BHP on {\sl any } open domains. As far as we know, there is no purely PDE proof of such a result. In a series of important works, De Silva and Savin \cite{deSilvaSavinAJM,deSilvaSavinJDE,deSilvaSavinRMI} initiated the derivation of BHP from only PDE tools, applying it to several free boundary problems. It is not clear to us how to make several of their arguments work in our general setting, though as alluded before, many of their strategy could be implement in a {\sl smoother} situation.  

Keeping in mind such a striking difference between the laplacian and the fractional laplacian, it would be desirable to have a BHP in our setting in the most general admissible domains. When one tries to follow the first strategy described above, one faces the obvious problem: a ball in $X_a$ is not the product of a ball in $X$ times an interval. In the case of very simple ambiant spaces (like Euclidean spaces) and Lipschitz domains  one can remedy to this issue by extending properly the function and using transformations which allow to apply BHP in so-called {\sl slit domains} (see e.g. \cite{Caffarelli_2007,ACS,CSS} and the papers of De Silva and Savin mentioned above).   

The notion of admissible domain we now consider is the one of {\sl uniform domains}.  If $x$ is a point in $\Omega$, denote by $\delta_\Omega(x) = d(x, X\backslash \Omega)$ the distance from $x$ to the boundary of $\Omega$. We define

\begin{definition}
Let $\gamma:[\alpha, \beta] \rightarrow \Omega$ be a rectifiable curve in $\Omega$ and let $c\in (0,1), C\in(1,\infty)$. We call $\gamma $ a uniform curve in $\Omega$ if 
$$
\delta_\Omega(\gamma(t)) \geq c\cdot \min\crl{d(\gamma(\alpha), \gamma(t)), d(\gamma(t), \gamma(\beta))}, \text{ for all } t\in[\alpha,\beta],
$$
and if 
$$\text{length}(\gamma) \leq C\cdot d(\gamma(\alpha), \gamma(\beta)).$$
The domain $\Omega$ is called uniform if any two points in $\Omega$ can be joined by a uniform curve in $\Omega$.
\end{definition}

The following result is proved in \cite{gianmarco}. 
\begin{lemma}\label{uniformNages}
The domain $X_a^+$ is uniform in $X_a$. 
\end{lemma}

The previous lemma is our starting point and combined with the results in \cite{lierl2014} for instance, allows to get 
\begin{lemma}\label{BHPextended}(Boundary Harnack inequality in extended spaces)
Assume that the Dirichlet space $(X_a:=X \times \mathbb R, \mathcal E_a, \mu_a)$ is our local Dirichlet space supporting an interior Harnack estimate for the generator $L_a$.   Let $\xi \in X$ and consider two nonnegative harmonic functions $u,v$ vanishing continuously on  $B_R(\xi)$ for some $R>0$ (this is a ball in $X_a$). Then there exists a constant $C>0$ such that for any $\xi \in X$
    $$
    \frac{u(x)}{u(x')}\leq C \frac{v(x)}{v(x')},$$
    
    for all $x, x'\in \overline B(\xi, r)$ and $r<R/2$. The constant $C$ depends only on the constant in the interior Harnack inequality. 
    \end{lemma}

    The point now would be to recover the desired BHP for $s-$harmonic functions on $X$ out of Lemma \ref{BHPextended} (which involves functions on $X_a$). One would need to prove that an $s-$harmonic function extends to a function satisfying the conditions of Lemma \ref{BHPextended} with a {\sl controled} geometry. As already mentioned, it is not clear in our general setup how to produce such a situation. We leave this problem to future work. 
    
{We would like to point out that Cao and Chen in \cite{caoChen} proved a Uniform Boundary Harnack Principle for operators of the type $(-L)^s$ {\sl without using any extension}, then addressing some questions above. However, as stressed in the previous discussion, we would like to have an approach based on the extension in view of some PDE problems posed on higher codimensional boundaries where the extension happens to be the original formulation.   }

\section*{Acknowledgments}
F.B. would like to thank Z.Q. Chen for discussions and pointing out  references.
F.B. was partially funded by NSF grant  DMS-2247117 when most of this work was completed. 
Q.L. would like to thank the University of Connecticut for its hospitality during the preparation of this work. Y.S. is partially funded by NSF DMS grant $2154219$, " Regularity {\sl vs} singularity formation in elliptic and parabolic equations".

\bibliographystyle{plain}
\bibliography{Refs.bib}

\end{document}